\documentclass[11pt,letterpaper,reqno]{amsart}
\usepackage[margin=2cm]{geometry}
\usepackage{hyperref,amssymb,amsfonts,amsthm,amsrefs,mathtools,verbatim,microtype,bm,nicefrac,tikz-cd,quiver}
\usepackage{stmaryrd}
\usepackage[shortlabels]{enumitem}
\usepackage{setspace}
\newtheorem{Theorem}{Theorem}[section]
\newtheorem{Proposition}[Theorem]{Proposition}
\newtheorem{Lemma}[Theorem]{Lemma}
\newtheorem{Corollary}[Theorem]{Corollary}

\newtheorem*{Claim*}{Claim}

\theoremstyle{definition}
\newtheorem{Definition}[Theorem]{Definition}

\newtheorem{Question}[Theorem]{Question}
\newtheorem{Observation}[Theorem]{Observation}

\DeclareFontFamily{U}{mathb}{\hyphenchar\font45}
\DeclareFontShape{U}{mathb}{m}{n}{
<-6> mathb5 <6-7> mathb6 <7-8> mathb7
<8-9> mathb8 <9-10> mathb9
<10-12> mathb10 <12-> mathb12
}{}
\DeclareSymbolFont{mathb}{U}{mathb}{m}{n}
\DeclareMathSymbol{\pprec}{\mathrel}{mathb}{"CE}
\DeclareMathSymbol{\ssucc}{\mathrel}{mathb}{"CF}

\DeclareFontFamily{U}{mathb}{\hyphenchar\font45}
\DeclareFontShape{U}{mathb}{m}{n}{
      <5> <6> <7> <8> <9> <10> gen * mathb
      <10.95> mathb10 <12> <14.4> <17.28> <20.74> <24.88> mathb12
      }{}
\DeclareSymbolFont{mathb}{U}{mathb}{m}{n}
\DeclareFontSubstitution{U}{mathb}{m}{n}
\DeclareMathSymbol{\monus}{2}{mathb}{"01}

\newcommand{\andd}{\wedge}

\newcommand{\da}{{\downarrow}}
\newcommand{\ua}{{\uparrow}}
\newcommand{\imp}{\rightarrow}
\newcommand{\Imp}{\Rightarrow}

\newcommand{\Biimp}{\Leftrightarrow}

\newcommand{\iso}{\cong}

\newcommand{\Nb}{\mathbb{N}}
\newcommand{\Zb}{\mathbb{Z}}
\newcommand{\Qb}{\mathbb{Q}}

\newcommand{\keq}{\simeq}

\newcommand{\mc}[1]{\mathcal{#1}}
\newcommand{\mf}[1]{\mathfrak{#1}}

\newcommand{\ol}[1]{\overline{#1}}

\usepackage{xcolor}
\usepackage{soul}

\AtBeginDocument{%
   \def\MR#1{}
}

\title{On cohesive products of fields} 

\author[Dimitrov]{Rumen Dimitrov}
\address{Department of Mathematics \& Philosophy\\
Western Illinois University\\
476 Morgan Hall\\
1 University Circle\\
Macomb, IL 61455\\
USA}
\email{rd-dimitrov@wiu.edu}
\urladdr{http://www.wiu.edu/users/rdd104/}

\author[Harizanov]{Valentina Harizanov}
\address{Department of Mathematics\\
The George Washington University\\
Phillips Hall\\
801 22\textsuperscript{nd} St.\ NW\\
Washington, DC 20052\\
USA
}
\email{harizanv@gwu.edu}
\urladdr{https://blogs.gwu.edu/harizanv/}

\author[Klatt]{Henry J. Klatt}
\address{The George Washington University\\
Phillips Hall\\
801 22\textsuperscript{nd} St.\ NW\\
Washington, DC 20052\\
USA}
\email{klatt@gwu.edu}
\urladdr{https://blogs.gwu.edu/klatt/}

\author[Srinivasan]{Keshav Srinivasan}
\address{Yeshiva University\\
Department of Mathematical Sciences \\
Yeshiva University \\
215 Lexington, New York, NY 10016\\
USA}
\email{ksrinivasan@gwu.edu}

\date{\today}

\begin{document}

\begin{abstract}
We develop the foundations of effective ultraproducts of fields and their Galois groups using the methods of computability theory. These computability-theoretic analogs of ultraproducts are called cohesive products, since the role of an ultrafilter is played by a cohesive set. A set of natural numbers is cohesive if it is infinite and cannot be partitioned into two infinite subsets by any computably enumerable set. In particular, we investigate the way cohesive products interact with field extensions with emphasis on both finite and infinite Galois extensions, and the associated Galois groups. We study the first-order theories and definability of cohesive powers of number fields, and characterize the infinite Galois groups of cohesive powers for a large class of infinite Galois extensions.  Finally, we introduce hyper-automorphisms, which are automorphisms of a cohesive power that respect non-standard field operations, and give a complete description of the hyper-automorphism groups of cohesive powers of a large class of computable Galois extensions, and use them to describe the classical infinite Galois groups of such fields.



\end{abstract}

\maketitle

\section{Introduction and preliminaries} \label{intro}

One of the most important constructions in model theory and other areas of mathematics is that of an ultraproduct. The ultraproduct construction uses a non-principal ultrafilter to combine a sequence of structures into a large structure with many desirable properties, and its first-order theory is controlled by the ultrafilter and the sequence of structures via {\L}o\'{s}'s theorem. However, the existence of a non-principal ultrafilter over an infinite set requires the use of a weak form of the axiom of choice, and the ultrapower construction typically produces an uncountable structure.  Both of these are obstacles to computable structure theory, which studies the algorithmic properties of countable structures. 

The first ultrapower-like construction is due to Skolem (\cite{Skolem}), who constructed a countable non-standard model of true arithmetic.  However, it more closely resembles a cohesive power than an ultrapower.  Skolem's model was constructed by taking the quotient of the set of all arithmetically definable functions via an arithmetically inseparable set (i.e., arithmetically cohesive set)
(see ~\cite{DimitrovHarizanovSKolem}).  

Feferman, Scott, and Tennenbaum investigated whether Skolem's construction can be made more effective by assuming that $C$ is only \emph{r-cohesive} (i.e., indecomposable for the collection of recursive sets) and by restricting the domain to total computable functions $f$.  They answered the question negatively by showing that it is not possible to obtain a model of Peano arithmetic in this way.  The result was announced in ~\cite{FefermanScottTennenbaum}, however the proof was never published.  Lerman (\cite{LermanCo-r-Max}) published a proof of the results of Feferman, Scott, and Tennenbaum. Additional results were obtained in ~\cites{HirschfeldModels, HirschfeldWheelerBook,ML1}.

Dimitrov (\cite{DimitrovCohPow}) generalized the Feferman, Scott, Tennenbaum construction from $\Nb$ to arbitrary computable structures.  This \emph{cohesive power} construction arose naturally from the study of the automorphisms of the lattice of computably enumerable vector spaces modulo finite dimension over a computable field (\cites{DimitrovLattices,DimitrovHarizanov,DimitrovHarizanovMillerMourad}).  In \cite{DimitrovCohPow},  Dimitrov showed that the cohesive power satisfies a restricted version of {\L}o\'{s}'s theorem.  The scope of {\L}o\'{s}'s theorem in the cohesive power can be extended by introducing extra decidability into the structure, and is fully restored if the structure has a computable elementary diagram.  Cohesive powers of other structures have been studied in \cites{HS,LinOrd,shafer2023effectivepowersomegadelta2}.

In this paper, we return to the study of cohesive products of computable fields. This paper is organized as follows.  In the remainder of this section, we discuss important definitions and results about computability theory and Galois theory, including splitting algorithms and effective Galois actions.

In section \ref{Q}, we present results about cohesive powers of the field of rational numbers.  In contrast to classical ultrapowers, cohesive powers of the field of rational numbers are rigid.  Further, we define a recursively axiomatizable first order analog of the categorical second order theory of the rational numbers, and demonstrate that the cohesive powers of $\Qb$ fail to satisfy this theory.

Section~\ref{alg} expands the analysis to algebraic number fields, which are finite extensions of the field of rational numbers.  We show that cohesive powers of Galois number fields preserve their Galois groups.  Further, we characterize their automorphisms and analyze their first order theories using the recent developments in Hilbert's 10th problem for number rings.  This answers question 6.4 of the AIM problem list for Definability and Decidability Problems in Number Theory.

Section~\ref{inf} analyzes cohesive powers of infinite algebraic extensions in general.  We classify the transcendence behavior of cohesive powers of field extensions over fields with splitting algorithms. In the second half of \ref{inf}, we develop a technique to extend infinite sequences of cohesive powers to hyper-finite lengths indexed over the cohesive power of $\Nb$.  Using this technique, we obtain a description of the infinite Galois groups of the algebraic parts of cohesive powers of infinite algebraic extensions by taking the projective limit of cohesive products of Galois actions. 

Finally in Section \ref{hyp}, we introduce the notion of hyper-automorphism groups.  A hyper-automorphism of a cohesive product of fields is an automorphism that preserves non-standard ``hyper-finite'' sums, products and polynomial functions.  We completely characterize the hyper-automorphism groups of cohesive products and powers of  ``fully computable'' Galois extensions in terms of the pseudo-finite cohesive products of finite Galois groups and their ``canonical hyper-finite'' sub extensions.  We equip the hyper-automorphism group with a non-standard metric, as in non-standard analysis, and analyze the groups relationship with both the topological and algebraic structure of the infinite Galois groups of the original field extensions, as well as the Galois groups of the their algebraic parts, re-characterizing the results of section \ref{inf}.

\subsection*{Computability Theory and Cohesive Products}

We use basic concepts and notation from computability theory and computable structure theory.  Comprehensive references include~\cites{LermanBook, SoareBookRE, SoareBookTC} for computability theory and~\cites{AshKnightBook,FHM, MontalbanBook,Montalban_2026} for computable structure theory.

Throughout the paper, we consider only first-order languages and first-order formulas. A computable language is a countable language with a uniformly computable set of symbols and their arities. Fix a computable language $\mf{L}$. A computable $\mf{L}$-structure $\mc{A}$ consists of a non-empty computable domain $A \subseteq \omega$ and a uniformly computable interpretation of all constant, function, and relation symbols of $\mf{L}$.  We often denote the domain of a structure $\mc{A}$ by $|\mc{A}|$, or simply $\mathcal{A}$ when there is no ambiguity. Equivalently, a computable structure is a structure with computable domain having a computable atomic diagram (or, equivalently, a computable quantifier-free diagram).  A sequence of structures $\{\mathcal{A}_i\}_{i\in \omega}$ is \emph{uniformly computable} if the sequence $\{ | \mathcal{A}_i |\}_{i \in \omega}$ of domains is uniformly computable, and there is a single binary computable function $f$ such that $f(i,\_)$ decides the atomic diagram of $\mathcal{A}_i$.

Now $\omega = \{ 0,1,2,\dots  \}$ is the set of natural numbers, while $\Nb = (\omega, +,\times, \leq )$ denotes the standard model of arithmetic.  We denote partial computable functions by $\varphi$, $\psi$, etc.  For a partial computable function $\varphi$, by $\varphi(n)\da$ we mean that $\varphi$ halts on input $n$, thus producing an output, and by $\varphi(n)\ua$, we mean that $\varphi$ does not halt on input $n$.  The notation $\varphi \keq \psi$ means that $\varphi$ and $\psi$ are equal partial functions:  for every $n$, either $\varphi(n)\da = \psi(n)\da$ or both $\varphi(n)\ua$ and $\psi(n)\ua$.  As usual, $\{{\varphi_e}\}_{e \in \omega}$ is the standard effective enumeration of all partial computable functions. 

Let $A,B \subseteq \omega$. Then we write  $A \subseteq^* B$ iff $A-B$  is finite, that is, $B$ contains all but finitely many elements of $A$. We also write  $B\supseteq^* A$ if $A\subseteq^* B$.
\begin{Definition}

   An infinite set $C \subseteq \omega$ is \emph{cohesive} if for every c.e.\ set $W$, either $C \subseteq^* W$ or $C \subseteq^* \ol{W}$.

\end{Definition}

More generally, a straightforward induction proof shows that if $C$ is a cohesive set and $X$ is a Boolean combination of c.e.\ sets, then either $C \subseteq^* X$ or $C \subseteq^* \ol{X}$.  Note that if $C$ is cohesive and $X$ is a Boolean combination of c.e.\ sets, then $C \cap X$ being infinite implies that $C \subseteq^* X$.  Also, if $A,B \supseteq^* C$, then $A\cap B \supseteq^*C $ as well.  

We will often say a property $P$ holds for almost every $i$ if $\{i \mid P(i)\} \supseteq^* C$.

\begin{Definition}\label{def-CohProd}
Let $\mf{L}$ be a computable language.  Let $\{\mc{A}_i\}_{i\in \omega}$ be a uniformly computable sequence of $\mf{L}$-structures with domains $\{|\mc{A}_i|\}_{i\in \omega}$.  Let $C \subseteq \omega$ be a cohesive set.  The \emph{cohesive product of $\{\mc{A}_i\}_{i\in \omega} $} over $C$ is the $\mf{L}$-structure $\prod_C \mc{A}_i$ defined as follows.

\begin{itemize}
\item Let $\mathcal{D}$ be the set of partial computable functions $\varphi$ such that $\forall i \, (\varphi(i)\da \,\Imp\, \varphi(i) \in |\mc{A}_i|)$ and $C \subseteq^* dom(\varphi)$.

\medskip

\item For $\varphi, \psi \in \mathcal{D}$, let $\varphi \simeq_C \psi$ denote $C \subseteq^* \{i \in \omega| \varphi(i)\da = \psi(i)\da\}$.  The relation $\simeq_C$ is an equivalence relation on $\mathcal{D}$.  Let $[\varphi]$ denote the equivalence class of $\varphi \in \mathcal{D}$ with respect to $\simeq_C$.

\medskip

\item The domain of $\prod_C \mc{A}_i$ is the set $|\prod_C \mc{A}_i| = \{[\varphi] \mid \varphi \in \mathcal{D}\}$.

\medskip

\item
Let $R$ be an $m$-ary relation symbol of $\mf{L}$.  For $[\varphi_{0}], \dots  , [\varphi_{m-1}] \in |\prod_C \mc{A}_i|$, define $R^{\prod_C \mc{A}_i}([\varphi_0], \dots  , [\varphi_{m-1}])$ by
\begin{align*}
R^{\prod_C \mc{A}_i}([\varphi_0], \dots  , [\varphi_{m-1}]) \;\Biimp\; C \subseteq^* \bigl\{i \in \omega \mid R^{\mc{A}_i}(\varphi_0(i), \dots  , \varphi_{m-1}(i))\bigr\}.
\end{align*}
Here we think of $R^{\mc{A}_i}(\varphi_0(i), \dots  , \varphi_{m-1}(i))$ as including the condition that $\varphi_j(i)\da$ for each $j < m$.

\medskip

\item Let $f$ be an $m$-ary function symbol of $\mf{L}$.  For $[\varphi_0], \dots  , [\varphi_{m-1}] \in |\prod_C \mc{A}_i|$, let $\psi$ be the partial computable function defined by
\begin{align*}
\psi(i) = f^{\mc{A}_i}(\varphi_0(i), \dots  , \varphi_{m-1}(i)),
\end{align*}
and note that $C \subseteq^* dom(\psi)$ because $C \subseteq^* dom(\varphi_j)$ for each $j < m$.  Define $f^{\prod_C \mc{A}_i}$ by
\begin{align*}
f^{\prod_C \mc{A}_i}([\varphi_0], \dots  , [\varphi_{m-1}]) = [\psi].
\end{align*}

\medskip

\item Let $c$ be a constant symbol of $\mf{L}$.  Let $\psi$ be the total computable function defined by $\psi(i) = c^{\mc{A}_i}$, and define $c^{\prod_C \mc{A}_i} = [\psi]$.
\end{itemize}

In the case where $\mc{A}_i$ is the same fixed computable structure $\mc{A}$ (with fixed presentation) for every $i$, the cohesive product $\prod_C \mc{A}_i$ is called the \emph{cohesive power of $\mc{A}$ over $C$}, and is denoted by $\prod_C \mc{A}$.  This is analogous to the ultrapower.
\end{Definition}

We often consider cohesive powers of computable structures by co-c.e.\ cohesive sets.  The co-c.e.\ cohesive sets are exactly the complements of the \emph{maximal} sets, which are the co-atoms in the lattice of c.e.\ sets modulo finite sets (see~\cite{SoareBookRE}*{Section~X.3}).  Such sets exist by a well-known theorem of Friedberg (see~\cite{SoareBookRE}*{Theorem~X.3.3}).  Cohesive powers are intended to be effective analogs of ultrapowers, so in light of this analogy, it makes sense to impose effectivity on the cohesive set (which plays the role of the ultrafilter), as well as on the base structure itself. It is often useful in a construction to be able to learn what numbers are not in the cohesive set $C$ when building a computable structure $\mc{A}$ so as to influence $\prod_C \mc{A}$ in a particular way.  Cohesive products by co-c.e.\ cohesive sets also have the property that every member of the cohesive product has a total computable representative.  See \cite{DimitrovHarizanovSurvey} for details.  

As in the case of ultrapowers, a computable structure $\mc{A}$ always naturally embeds into its cohesive power.  For $a \in |\mc{A}|$, let $f_a: \omega \to |\mc{A}|$ be the total computable function with constant value $a$.  Then for any cohesive set $C$, the map $a \mapsto [f_a]$ embeds $\mc{A}$ into $\prod_C \mc{A}$.  This map is called the \emph{canonical embedding} of $\mc{A}$ into $\prod_C \mc{A}$, and is a two-quantifier elementary embedding. 

If $\mc{A}$ is finite and $C$ is cohesive, then every partial computable function $\varphi \colon \omega \imp |\mc{A}|$ with $C \subseteq^* dom(\varphi)$ is eventually constant on $C$.  In this case, every element of $\prod_C \mc{A}$ is in the range of the canonical embedding, and therefore $\mc{A} \iso \prod_C \mc{A}$.  If $\mc{A}$ is an infinite computable structure, then every cohesive power $\prod_C \mc{A}$ is countably infinite: infinite because $\mc{A}$ embeds into $\prod_C \mc{A}$, and countable because the elements of $\prod_C \mc{A}$ are represented by partial computable functions, of which there are only countably many.  Furthermore, in the infinite case, the canonical embedding is never onto, even if the cohesive power is isomorphic to the original structure. For example, a cohesive power of a dense linear order without endpoints is isomorphic to the original structure, but the image of the canonical embedding is a bounded set.  See~\cite{DimitrovCohPow} for further details.

\subsection*{Analogs of {\L}o\'{s}'s theorem}  \hspace{1mm}

A restricted version of {\L}o\'{s}'s theorem holds for cohesive powers. 

\begin{Theorem}[The Fundamental Theorem of Cohesive Products]\label{thm-LosProdParam}
Let $\mf{L}$ be a computable language, let $\{\mc{A}_i\}_{i\in\omega}$ be a uniformly computable sequence of $\mf{L}$-structures, and let $C$ be a cohesive set.
\begin{enumerate}[(1)]
\item\label{it-LosProdPramSig2} Let $\Phi(x_0, \dots  , x_{m-1})$ be a $\Sigma^0_{2}$ formula.  Then for any $[\varphi_0], \dots  , [\varphi_{m-1}] \in |\prod_C \mc{A}_i|$,
\begin{align*}
\prod\nolimits_C \mc{A}_i \models \Phi([\varphi_0], \dots  , [\varphi_{m-1}]) \quad\Imp\quad C \subseteq^* \bigl\{i\in \omega \mid \mc{A}_i \models \Phi(\varphi_0(i), \dots  , \varphi_{m-1}(i))\bigr\}.
\end{align*}

\medskip

\item\label{it-LosProdPramPi2} Let $\Phi(x_0, \dots  , x_{m-1})$ be a $\Pi^0_{2}$ formula.  Then for any $[\varphi_0], \dots  , [\varphi_{m-1}] \in |\prod_C \mc{A}_i|$,
\begin{align*}
C \subseteq^* \bigl\{i \in \omega \mid \mc{A}_i \models \Phi(\varphi_0(i), \dots  , \varphi_{m-1}(i))\bigr\} \quad\Imp\quad \prod\nolimits_C \mc{A}_i \models \Phi([\varphi_0], \dots  , [\varphi_{m-1}]).
\end{align*}

\end{enumerate}
\end{Theorem}

It turns out that this result is tight in general.  We will now switch from cohesive products to cohesive powers. The next theorem is a version of {\L}o\'{s}'s theorem for cohesive powers (see \cite{CohPowCiE}).

\begin{Theorem}[The Fundamental Theorem of Cohesive Powers]\label{thm-LosGeneral}
Let $\mf{L}$ be a computable language, let $\mc{A}$ be an computable $\mf{L}$-structure, and let $C$ be a cohesive set.

\begin{enumerate}[(1)]
\item\label{it-LosDelta2Param} Let $\Phi(x_0, \dots  , x_{m-1})$ be a Boolean combination of $\Sigma_{1}^0$ and  $\Pi_1^0$ formulas.  Then for any $[\varphi_0], \dots  , [\varphi_{m-1}] \in |\prod_C \mc{A}|$,
\begin{align*}
\prod\nolimits_C \mc{A} \models \Phi([\varphi_0], \dots  , [\varphi_{m-1}]) \quad\Biimp\quad C \subseteq^* \bigl\{i \in \omega \mid \mc{A} \models \Phi(\varphi_0(i), \dots  , \varphi_{m-1}(i))\bigr\}.
\end{align*}

\medskip

\item\label{it-LosDelta3Sent} Let $\Phi$ be a Boolean  combination of $\Sigma^0_{2}$ and $\Pi_2^0$  sentences.  We have $\mc{A} \models \Phi$ if and only if $\prod_C \mc{A} \models \Phi$.

\medskip

\item\label{it-LosSigma3Sent} Let $\Phi$ be a $\Sigma^0_{3}$ sentence.  If $\mc{A} \models \Phi$, then $\prod_C \mc{A} \models \Phi$.
\end{enumerate}
\end{Theorem}

The first studied and most central cohesive power is the cohesive power of the standard model of arithmetic.  In the case where $\mathcal{A} = \Nb$, there is a $\Pi_3^0$ sentence $\Psi$ with $\Nb \models \Psi$ but $\prod_C \Nb \models \neg \Psi$ (see \cites{LermanCo-r-Max, FefermanScottTennenbaum}). This shows that Theorem \ref{thm-LosGeneral}  is tight in general.

We recover full {\L}o\'{s}'s theorem for all first-order sentences when we consider cohesive powers of decidable structures (see \cite{DimitrovCohPow}), that is, structures in which one may decide the entire elementary diagram algorithmically. One may see this by adding predicates for every first-order definable relation.

In Definition~\ref{def-CohProd},  the condition $\forall i \, (\varphi(i)\da \,\Imp\, \varphi(i) \in |\mc{A}_i|)$ and $C \subseteq^* dom(\varphi)$ is equivalent to the a priori slightly weaker $\{ i \in \omega \mid \varphi(i)\da \,\andd\, \varphi(i) \in |\mc{A}_i|\}\supseteq^* C$.  If $\varphi$ is a partial computable function satisfying the weaker condition, let $\psi$ be the partial computable function given by
\begin{align*}
\psi(i) =
\begin{cases}
\varphi(i) & \text{if $\varphi(i)\da$ and $\varphi(i) \in |\mc{A}_i|$}\\
\ua & \text{otherwise}.
\end{cases}
\end{align*}
Then $\psi$ satisfies the original condition, and $\psi \simeq_C \varphi$.

We may also relax the definition to allow $\{ \mc{A}_i \}_{ i \in W}$ to be a partial sequence, with domain a c.e.\ set $W$ with $W \supseteq^* C$.  Since this particular generalization is new, we prove that the fundamental theorem of cohesive products remains true in this more general setting.

\begin{Proposition}
    Let $\{\mc{A}_i\}_{i \in W }$
    be a uniformly computable partial sequence of structures, $C$ a cohesive set with $C\subseteq^* W$, and let $f: \omega\to W$ be a computable bijection.  Then there is a cohesive set $D$ such that $\prod_C \mc{A}_i \cong \prod_D \mc{A}_{f(i)}$. 
\end{Proposition}

\begin{proof}
    Let $D= f^{-1}(C)$.  First, we claim that $D$ is cohesive.  Let $V$ be a c.e. set with $V \cap D $ infinite.  Then $f(D) \cap f(V) = C\cap f(V)$ is infinite.  Thus, $C \subseteq^* f(V)$.  Therefore $ D \subseteq^* V$.

    It is clear that for any c.e. set $V$, $D\subseteq^* V \iff C \subseteq^* f(V).$

    Let $\Psi : \prod_C \mathcal{A}_i \to \prod_D \mathcal{A}_{f(i)}$ be defined by $\ [h]\mapsto [h\circ f] $.  Let $\Phi(x_1,\dots  ,x_n)$ be a $\Sigma^0_1$ formula, and let $[h_1],\dots  ,[h_n] \in \prod_C \mathcal{A}_i$. Then

$$\prod_C \mathcal{A}_i \models \Phi([h_1],\dots  ,[h_n]) \iff C \subseteq^*\{ i \mid \mathcal{A}_i \models \Phi(h_1(i),\dots  ,h_n(i))\} $$$$ \iff  D\subseteq^* f^{-1}(\{ i \mid \mathcal{A}_i \models \Phi(h_1(i),\dots  ,h_n(i))\} ) \iff D \subseteq^* \{ i \mid \mathcal{A}_{f(i)} \models \Phi (h_1(f(i),\dots  ,h_n(f(i)) ) \} $$$$\iff \prod_D \mathcal{A}_{f(i)} \models \Phi ( \Psi ([h_1 \circ f]),\dots  ,\Psi([h_n\circ f]))$$
In particular, every function, constant, and relation is preserved.
\end{proof}

For this paper, an important corollary to the fundamental theorem of cohesive powers is the following. 

\begin{Corollary} \label{One Quant Definitions Lift}
    If $\mc{A}$ is a $\Sigma_1^0$ $(\Pi_1^0)$ definable substructure of $\mc{B}$ (with or without parameters), then $\prod_C \mc{A}$ is $\Sigma_1^0$ ($\Pi_1^0)$ definable in $\prod_C \mc{B}$, using the same formula and parameters. 
    
\end{Corollary}

\begin{proof}
    Suppose $\mc{A} = \{ b \in B \mid \mc{B} \models \Psi (b,\vec{p}) \}$, where $ \Psi$ is $\Sigma_1^0$ ($\Pi_1^0)$ and $\vec{p} \in B$. Then 

    $\prod_C \mc{A} = \{ [\varphi_e] \mid \text{ran}(\varphi_e) \subseteq \mc{A}  \wedge \text{dom} (\varphi_e) \supseteq^* C \} = \{ [\varphi_e ] \mid     \mc{B}\models \Psi (\varphi_e(i), \vec{p}) \text{ for almost every } i\} .$

    $ =\{ [\varphi_e ] \mid \prod_C \mc{B} \models \Psi([\varphi_e],\vec{p}) \}$

    Thus, $\Psi $ defines $\prod_C \mc{A}$ in $\prod_C \mc{B}$. 
\end{proof}

\subsection{Fields and Galois Theory}

For completeness, we review some of the basics of field theory and Galois theory. A field extension $K/F$ is a pair of fields $F,K$ such that $F \subseteq K$.  The extension $K/F$ is \textit{finite} (\textit{infinite}) if $K$ is a finite (infinite) dimensional $F$-vector space.  The dimension of $K$ as a vector space is known as the \textit{degree} of the extension, and is denoted $[K:F]$. For $E \subseteq F \subseteq K$ we have $[K:E]=[K:F][F:E]$.

An element $\alpha \in K$ is \textit{algebraic} over $F$ if $\alpha$ is a root of some polynomial $p\in F[x]$.  An element $\alpha \in K$ is called \textit{transcendental} over $F$ if it is not algebraic over $F$.  A finite set of elements $\{\alpha_1,\dots  ,\alpha_n\} \subseteq K$ is called \textit{algebraically independent} if $\alpha_i$ is transcendental over $F (\alpha_1,\dots  ,\alpha_{i-1} ,\alpha_{i+1} ,\dots  , \alpha_n)$ for each $i \leq n$, where $F (\beta_0,\dots  ,\beta_n)$ is the smallest subfield of $K$ extending $F$ and containing $\beta_0,\dots  ,\beta_n$.  An infinite set $\{ \alpha_1,\alpha_2,\dots  \}  \subseteq K$ is algebraically independent if every finite subset is algebraically independent.  The \textit{transcendence degree} of $K/F$ is the cardinality of a maximal algebraically independent subset of $K$.  A maximal algebraically independent set exists by Zorn's Lemma.

The algebraic part of $K/F$, denoted $\text{Alg}_F(K)$, is the field of all elements of $K$ that are algebraic over $F$.  When the base field is clear from context, we will simply write $\text{Alg}(K)$. For every $\alpha \in K$ that is algebraic over $F$, there is a unique monic polynomial $p_\alpha \in F[x]$ of minimal degree such that $p_\alpha (\alpha ) = 0$, referred to as the minimal polynomial of $\alpha$.  Moreover, for every other $q(x) \in F[x]$, if $q(\alpha)=0$, then $p_\alpha \mid q$.  A polynomial $p\in F[x]$ is \textit{separable} if for every field extension $K/F$ in which $p$ factors in $K[x]$ into linear factors, $p$ has no repeated roots in $K$,  i.e., $p(x) = (x-\alpha_1)\cdots(x-\alpha_n)$ where $\alpha_i \not= \alpha_j$ for every $i\not= j$. 

An extension $K/F$ is \textit{algebraic} if every element of $K$ is algebraic over $F$.  It is \textit{normal} if it is algebraic and every irreducible $p(x)\in F[x]$ with a root in $K$ factors completely into linear factors in $K[x]$,  and it is \textit{separable} if every $\alpha \in K$ has a separable minimal polynomial.  All extensions of characteristic $0$ are separable.  An extension is \textit{Galois} if it is both normal and separable.

If $\alpha$ is algebraic over $F$, then the degree of $\alpha$ over $F$ is the degree of its minimal polynomial $p_\alpha \in F[x]$. We will denote the degree by $\deg_F (\alpha)$, or by $\deg(\alpha)$ when it is clear from context which field we are taking the degree over.  Further, the \textit{simple extension} of $F$ by $\alpha$, denoted $F(\alpha)$, is the least extension of $F$ containing $\alpha$.  It is always the case that $[F(\alpha) :F] = \deg_F (\alpha)$.  An element $ b \in F(\alpha)$ can be written uniquely as $a_{n-1} \alpha^{n-1} +\cdots+ a_0$, where $a_i \in F$, and $n= \deg_F (\alpha)$.  In other words, $1, \alpha, \dots  , \alpha^{n-1}$ is an $F$-vector space basis of $F(\alpha)$.  

The finite simple extensions are completely classified up to isomorphism.

 \begin{Theorem}[Primitive Element Theorem]
     Let $K/F$ be a finite and separable extension.  Then $K = F(\alpha)$ for some $\alpha \in K$.
 \end{Theorem}

Since all Galois extensions are separable, all finite Galois extensions are of the form $F(\alpha) / F$ for some $\alpha$.  A finite separable extension $F(\alpha)/F$ is Galois if and only if the minimal polynomial $p_\alpha$ of $\alpha$ splits into linear factors in $F(\alpha)$.

    The group $\text{Aut}(K/F)$ is the subgroup of $\text{Aut} (K)$ that fixes $F$, i.e., for every $\sigma \in \text{Aut}(K/F)$ we have $\sigma(\alpha) = \alpha$ for every $\alpha \in F$. 
    Often, we will omit parentheses and write $\sigma \alpha = \sigma(\alpha)$.  We will denote the action of $Aut(K/F)$ on $K$, that is the map $Aut(K/F) \times K \to K$  defined by $(\sigma, \alpha) \mapsto \sigma(\alpha)$, by $Aut(K/F) \curvearrowright K$.

    If $K/F$ is Galois, then $\text{Aut} (K/F)$ is denoted $\text{Gal}(K/F)$ and is called the \textit{Galois group} of $K/F$. The structure of the Galois group reveals a great deal of the arithmetic of the field extension via the fundamental theorem of Galois theory.

\begin{Theorem}[Fundamental Theorem of Finite Galois Theory] \label{fundamental theorem of Galois theory}

Let $K/F$ be a finite Galois extension.   

\begin{itemize}
    \item For each subgroup $H \subseteq Gal (K/F)$, there is a subfield $ E^H $ of $ K$ that is the largest subfield fixed by $H$. That is,

    \subitem $E^H = \{ \alpha \in K \mid \forall \sigma \in H (\sigma \alpha = \alpha) \}$.

    \item For each field $E$ with $F \subseteq E \subseteq K$, we have that $K/E$ is Galois, and $Gal(K/E)$ is a subgroup of $Gal (K/F)$.  

    \item The above two operations are inverse to each other. This is known as the \textit{Galois correspondence}:

    \subitem $Gal( E^H /F) = H$ and $E^{Gal(L/F )} = L$.

    \item The Galois correspondence is inclusion-reversing:  

        \subitem $H \subseteq H'$ if and only if $E^{H} \supseteq E^{H'}$.  

    \item $ | H| = [K :E^H]  $ and $ |Gal(K/F) : H | = [ E^H :F]$.

    \item A group $H$ is a normal subgroup of $Gal(K/F)$ if and only if $E^H /F$ is Galois.  In this case, $Gal(K/F)/H \cong Gal(E^H /F)$, with isomorphism induced by the quotient map $Gal (K/F) \to Gal(E^H/F)$ that sends $\sigma \in Gal( K/F) $ to $\sigma \upharpoonright E^H \in Gal (E^H / F) $.
\end{itemize}
    
\end{Theorem}
    
For more about Galois theory see \cites{dummit_foote_2004,milne2022}.

The infinite version of the fundamental theorem of Galois theory requires a skosh of topology. 

\begin{Definition}
    A topological group is a topological space $G$, together with continuous operations $*:G^2\to G$ and $^{-1}: G\to G$, such that $(G,*,^{-1} )$ is a group.  
\end{Definition}

\begin{Definition}
\hfill
  \begin{itemize} 
  
  \item Let $(L, <)$ be a linear order. An inverse system of groups is a collection of groups $\{G_i\}_{i\in L}$, together with a collection of surjective group homomorphism $f_{m,n} : G_{m} \to G_n$, whenever $m>n$.  The group maps must satisfy $f_{n,k} \circ f_{m,n}=f_{m,k}$ whenever $k<n<m$. 

   \item Each $G_i$ is a topological group (typically with the discrete topology), and the maps $f_{n,m}$ are continuous.

   \item  The \emph{projective limit} of the inverse system is the group
   
   $$\varprojlim_{i\in L} G_i = \{\{ g_i\}_{i\in L }\in \prod_{i\in L} G_i \mid \forall m>n  ((f_{m,n} (g_m )=g_n))  \}$$

  \item  The projective limit  $\varprojlim_{i\in L} G_i$ is given the subspace topology inherited from $\prod_{i\in L} G_i$. \end{itemize}

\end{Definition}

Confusingly, the projective limit is also known as both the limit and inverse limit. When all $G_i$ are finite and given the discrete topology, for each $h\in G_j$, the set $U_{h,i} = \{ \{ g_i\}_{i\in L} \mid g_j =h \}$ is both open and closed, and the collection of all such $U_{h,j}$ generates the topology on $G$.  In this case, $\varprojlim _{i\in L }G_i$ is known as a \emph{profinite group}.

We turn our attention back to field extensions.

\begin{Lemma}
    An infinite extension $K/F$ is Galois if and only if $K = \bigcup_{i\in I} K_i$ for some $\{ K_i \}_ {i\in I}$, where each $K_i /F$ is a finite Galois extension.  
\end{Lemma}

We are primarily interested in the case where $K= \bigcup_{i\in \omega} K_i$, with each $K_i$ finite, and $K_i \subseteq K_{i+1}$.  This is no less general for countable fields, as we may take $L_i = \langle K_1,\dots  ,K_i \rangle$.  That is, we may take $L_i$ to be the least subfield of $K$ containing $K_1,\dots  ,K_i$.

\begin{Lemma}
    Let $K/F$ be a  extension such that there exist finite  sub-extensions $K_i/F$, with $K=\bigcup_{i\in \omega} K_i$ and $K_i \subseteq K_{i+1}$.  Let $G_i = Gal(K_i /F)$, and let $f_{i+1} : G_{i+1}\to G_i$ be the associated surjections as in the fundamental theorem of Galois theory. Then $Gal (K/F) \iso \varprojlim G_i$, where $G_i$ is given the discrete topology.  In particular, $Gal (K/F)$ is a topological group, with basic open sets of the form $U_{\sigma, i} = \{ f \in Gal (K/F) \mid f \upharpoonright K_i = \sigma \}$ for $\sigma \in G_i$.  
\end{Lemma}

While the setting in which each $K_i /F$ is finite is fully general, we will encounter situations in which we are forced to consider infinite Galois $K_i/F$ indexed over some linear order $L\not=\omega$. 

\begin{Lemma} \label{Infinite Galois Theory infinite proj lim}
    Let $K/F$ be a Galois extension such that there exists sub-extensions $K_i/F$, indexed over some linear order $L$ with $K = \bigcup_{i\in L} K_i$ and $K_i \subseteq K_{j}$ whenever $i\leq j$.  Let $G_i = Gal(K_i /F)$, and  for any $j>i$, let $f_{j,i} : G_{j}\to G_i$ be the associated surjections as in the fundamental theorem of Galois theory. Then $Gal (K/F) \iso \varprojlim G_i$.
\end{Lemma}

With the profinite topology in hand, we may now state the fundamental theorem of Galois theory in the case where $K/F$ is infinite.  The only real difference is that now the subgroups in the correspondence are \textit{closed} in the profinite topology on $Gal(K/F)$.

\begin{Theorem}[Fundamental Theorem of Infinite Galois Theory]

\hspace{1mm}

    \begin{itemize}
    \item For each \emph{closed} subgroup $H \subseteq Gal (K/F)$, there is a subfield $ E^H $ of $ K$ that is fixed by $H$.

    \subitem That is, $E^H = \{ \alpha \in K \mid (\forall \sigma \in H) (\sigma \alpha = \alpha) \}$.

    \item For each field $E$ with $F \subseteq E \subseteq K$, and $K/E$ is Galois, and $Gal(K/E)$ is a \emph{closed} subgroup of $Gal (K/F)$.  

    \item The above two operations are inverse to each other. This is known as the \textit{Galois correspondence}:

    \subitem $Gal( E^H /F) = H$ and $E^{Gal(L/F )} = L$.

    \item The correspondence is inclusion reversing:  

        \subitem for closed $H, H'$, we have $H \subseteq H'$ if and only if $E^{H} \supseteq E^{H'}$.  

    \item For \emph{closed} $H$, $ | H| = [K :E^H]  $, and $ |Gal(K/F) : H | = [ E^H :F]$

    \item For \emph{closed} subgroups $H$ of $Gal(K/F)$, H is a normal subgroup of $Gal(K/F)$, if and only if $E^H /F$ is Galois.  In this case, $G/H \cong Gal(E^H /F)$.
\end{itemize}
\end{Theorem}

The behavior of non-closed subgroups within the Galois correspondence is determined by their closure.  

\begin{Lemma} \label{same fixed field iff same closure}
    Let $K/F$ be an infinite Galois extension, and let $H,G$ be subgroups of $ Gal (K/F)$, not necessarily closed.  $E^H = E^G $ if and only if $\overline{H} = \overline{G}$.  That is to say two groups fix the same field if and only if they have the same closure. 

    In particular, a subgroup $G$ of $Gal (K/F)$ is dense if and only if $E^G = F$.  
\end{Lemma}

\subsection{Computability of Field Extensions}

There is some variability in how one can computably present a field extension.  It is common to present $K/F$ as a computable presentation of $K$, with a predicate for $F$.  However, it is more general to present $K/F$ as a triple $(K,F,\Psi)$, where $K$ and $F$ are computably presented with domain $\omega$, and $\Psi : F\to K$ is a computable ring homomorphism.  The difference is subtle, but has important ramifications for infinite extensions and uniformity considerations.  In this paper, we will exclusively use the more general presentation.  We will also often use uniformly computable sequences of field extensions, which we define here formally for completeness.

\begin{Definition}
        A sequence of field extensions $(K_i,F_i, \Psi_i)$ is uniformly computable if there is are three computable functions $f,g,h$ of two variables such that $f(i,\_)$ decides the atomic diagram of $K_i$, and $g(i,\_)$ decides the atomic diagram of $F_i$, and $h(i,\_)$ is the map $\Psi_i$. 
\end{Definition}

An important bifurcation in the computability of fields is whether or not a field $F$ has a \textit{splitting algorithm}.  A splitting algorithm for a field $F$ is a computable procedure which, given $p\in F[x]$, factors $p$ into a product of irreducible factors.  It is sufficient for the set of irreducible polynomials in $F[x]$ to have a decision procedure.
A root algorithm for $F$ decides whether a polynomial $p\in F[x]$ has a zero in $F$. A computable field has a splitting algorithm if and only if it has a root algorithm  (see \cite{10.1098/rsta.1956.0003}).

Many familiar fields have splitting algorithms.  Kronecker developed a splitting algorithm for $\Qb$. Every finite extension of $\Qb$ has a splitting algorithm. There are splitting algorithms for the global fields $\mathbb{F}_p(t)$ (\cite{POHST2005617}). In general, if a computable field $F$ has a splitting algorithm, and $t$ is transcendental over $F$, then $F(t)$ also has a splitting algorithm. If $
\alpha$ is algebraic and separable over a computable field $F$ with a splitting algorithm, then $F(\alpha)$ has a splitting algorithm.  

Not all computable fields have splitting algorithms.  The field $\Qb(\sqrt{p_e})_{e\in H}$, where $H = \{ e \mid \varphi_e(e)\downarrow \} $ is the halting set and $\{p_e \}_{e\in \omega}$ is the sequence of prime numbers, can be computably presented since $H$ is a c.e.\ set.  However, no computable isomorphic copy of this field can have a splitting algorithm, as otherwise deciding if $x^2 -p_e$ factors would allow one to decide whether $e\in H$, contradicting the fact that $H$ is not computable.

Rabin proved the following important result in \cite{ce2fcec9-0c83-3872-8226-7a164e94ad51}.

\begin{Theorem}[Rabin]
   Given a computable field $F$, the extension $\overline{F}/F$ admits a computable presentation  $(\overline{F},F,\Psi)$.  That is, $\overline{F}$ can be computably presented, and $F$ can be computably embedded into $\overline{F}$. Moreover, the copy of $F$ in $\overline{F}$ (that is, the image of $\Psi$) is a computable subset of $\overline{F}$ if and only if $F$ has a splitting algorithm.  
\end{Theorem}

We will also make frequent use of the degree of an element of a field extension.  The degree map also need not be computable.  However, the map is computable in two key settings.

\begin{Proposition}
    If $K/F$ is a computable field extension, and $F$ has a splitting algorithm, then $\deg_F :K \to \Nb$ is computable. 
\end{Proposition}

    \begin{proof}
        Let $\alpha$ be an element of $K$.  Search through $F[x]$ for a polynomial $q(x)$ with $q(\alpha)=0$.  Run the splitting algorithm on $q$ to factor $q$ into irreducible factors.  One factor must be the minimal polynomial $p_\alpha$, which can be identified by evaluating each factor on $\alpha$.  Once $p_\alpha$ has been identified, output its degree. 
    \end{proof}

\begin{Proposition}
    If $K/F$ is a finite Galois extension with Galois group $G$, the degree map is computable from the Galois action $G \curvearrowright K$.  
\end{Proposition}

\begin{proof}
    Let $\alpha \in K$, and let $G= \{ g_0, g_1,\dots  ,g_n\}$. Recall that $\deg_F(\alpha) = [ F(\alpha): F ]$. The field $F(\alpha)$ has $F$ basis $\{ g_0 \alpha, g_1\alpha,\dots, g_n \alpha\}$. (Note that this enumeration will in general not be repetition-free.)  Therefore, $\deg_F (\alpha ) =|\{ g_0 \alpha, g_1 \alpha,\dots, g_n \alpha \}|$.
\end{proof}

\begin{Corollary}
    If $K/F$ is finite and Galois, a primitive element $\alpha \in K$ can be found uniformly from the Galois group and action $G \curvearrowright K$. 
\end{Corollary}

\begin{proof}

For $\alpha \in K$, we have $F(\alpha) = K$ if and only if $\deg_F (\alpha) = |G|$.
\end{proof}

\begin{Proposition}\label{tfaegal}
    Let $K/F$ be a computable finite Galois extension.  Each of the following may be computed uniformly from the others: 

    \begin{enumerate}
        \item The Galois group and Galois action $Gal(K/F)\curvearrowright K$;

        \item The minimal polynomial of a primitive generator $\alpha \in K$;

        \item The degree map $\deg_F : K \to \Nb$, and the degree $[K:F]$;

        \item The minimal polynomial map $ K \to F[x]$ and the degree $[K:F]$.
    \end{enumerate}
    
\end{Proposition}

\begin{proof}
    (1) $\Imp$ (2)  An element $\alpha \in K$ is a primitive generator if and only if $g\alpha \not= \alpha$ for all non-trivial $g \in Gal(K/F)$.  Once such an $\alpha$ is identified, we may compute its minimal polynomial by recalling that $p_\alpha = \prod_{g\in G}(x- g\alpha)$.

    (2) $\Imp$ (3) Let $p$ be the minimal polynomial for a generator of $K$.  Note that $\deg(p) = [K:F]$. Search through $K$ for all roots of $p$, call them $\alpha_1,\dots  ,\alpha_n$.  Let $\beta \in K$.  Since $\alpha_1$ is a generator, $\beta = a_{n-1} \alpha_1^{n-1} +\cdots+a_0$.  With these coefficients in hand, compute $\beta_j=a_{n-1} \alpha_j^{n-1} +\cdots+a_0$ for $j \leq n$. Then   $\beta=\beta_1,\beta_2,\dots  ,\beta_n$ is a list of the Galois conjugates of $\beta$, not necessarily repetition-free.  The number of unique entries in the list is the degree of $\beta$.

    (3) $\Imp$ (4)  Let $\beta \in K$, and let $m = \deg_F (\beta)$.  The unique polynomial $p$ with $\deg(p) = \deg_F(\beta)$ and $p (\beta) = 0$ is the minimal polynomial of $\beta$. 

    (4) $\Imp$ (1) Search for $\alpha_1 \in K$ where $\deg(p_{\alpha_1}) = [K:F]$. The element  $\alpha_1$ is necessarily a generator of $K$.  Let $\alpha_1, \alpha_2,\dots  , \alpha_n$ be a repetition-free list of the roots of $p_{\alpha_1}$, i.e., $\alpha_1$'s Galois conjugates. The group  $Gal(K/F)$ is isomorphic to a subgroup of the symmetric group $S_n$.  We may identify which subgroup it is as follows. Since $\alpha_1$ generates $K$, we have $\alpha_j = p_j (\alpha_1)$ for some $p \in F[x]$.  Let $\sigma_i\in S_n$ such that $\alpha_{\sigma_i (j)} = p_j(\alpha_i)$.  Then $Gal(K/F) \cong \{ \sigma_1,\dots  \sigma_n\} \leq S_n$.  

    The action $Gal(K/F) \curvearrowright K$ is computed by setting $\sigma_i \beta = q(\alpha_i)$, where $q\in F[x]$ is such that $\beta= q(\alpha_1)$.
\end{proof}

\section{Cohesive powers of  the field of rational numbers}\label{Q}

We begin by recalling the previous work on the cohesive powers of $\Qb$ from \cite{DimitrovHarizanovMillerMourad}. Consider the field of rational numbers, $(\mathbb{Q} , + , \cdot).$ Recall that a set of natural numbers is maximal if it is c.e.\ and its complement is cohesive.  Equivalently, a c.e.\ set $A$ is maximal if for any other c.e.\ set $B$, if $A\subseteq^* B$ and $B-A$ is infinite, then $B$ is co-finite.

For any two maximal sets $M_{1}$ and $M_{2}$ with complements $\overline{M_{1}}$ and $\overline{M_{2}}$, we have
\begin{equation*}\prod _{\overline{M_{1}}}(\mathbb{Q} , + , \cdot ) \cong \prod _{\overline{M_{2}}}(\mathbb{Q} , + , \cdot )~ \mathrm{i} \mathrm{f} \mathrm{f}~ M_{1}  \equiv _{m}M_{2}\text{,}
\end{equation*}where $ \equiv _{m}$ stands for $m$-equivalence. A set $X$ is $m$-reducible (many-one and not necessarily one-one reducible) to a set $Y$ if there is a computable reduction from $X$ to $Y$, that is, a function $f$ on the set of natural numbers such that $ \forall i[i \in X \Leftrightarrow f(i) \in Y]$. The sets $X$ and $Y$ are $m$-equivalent, in symbols $ X  \equiv _{m}Y $, if they are $m$-reducible to each other.

A cohesive power $\prod_C \Qb $ is always a transcendental extension of $\Qb$ with infinite transcendence degree, and no algebraic part. 

\begin{Proposition}(\cite{DimitrovHarizanovMillerMourad}) A cohesive power 
    $\prod_C \Qb$ does not admit any non-trivial automorphisms, i.e., it is rigid. 
\end{Proposition}

  This is in contrast to the ultrapower, which is saturated, and therefore admits $2^{2^{\aleph_0}}$ automorphisms (see \cite{Keisler}).

\begin{Proposition}(\cite{DimitrovHarizanovMillerMourad}) A cohesive power $\prod_C \Qb$ is not elementarily equivalent to $\Qb$.
    
\end{Proposition}

This raises the question whether $\prod_C \Qb$ satisfies some reasonable fragment of the theory of $\Qb$.  We develop such a fragment $T$, but find that $\prod_C \Qb \not \models T$.

The field $\mathbb{Q}$ is the unique model up to isomorphism of the
second-order theory consisting of the field axioms, axioms of the form
\(\neg(1 + \cdots + 1 = 0)\) for any number of \(1\)'s, and the axiom 
\begin{equation*}\
    \forall X  SubField(X) \implies \forall y(y\in X),
\end{equation*}\ where $SubField(X)$ is the abbreviation for 
$$ (0 ,1 \in X )\wedge \forall x \forall y \forall z [( x,y \in X \wedge x+z = y)\implies \in X]\wedge \forall x\forall y \forall z ([( x,  y \in X \wedge x\not=0 \wedge xz = y )\implies z \in X)].$$ This is a model-theoretic rephrasing of the standard fact that $\Qb$ is the unique prime field of characteristic $0$.  

\begin{Definition} The theory \(T\) is the first-order theory consisting of the field axioms, axioms of the form \(\neg(1 + \cdots + 1 = 0)\) for any number of \(1\)s, and the axiom schema 
$$ SubField(\Phi) \implies \forall y \Phi(y)$$
for each formula $\Phi(x)$ with parameters, where $SubField(\Phi)$ is an abbreviation for 
$$ \Phi(0) \wedge \Phi(1) \wedge \forall x \forall y \forall z [( \Phi (x) \wedge \Phi (y) \wedge x+z = y)\implies  \Phi(z)]\wedge \forall x\forall y \forall z [( \Phi(x) \wedge \Phi(y) \wedge xz = y \wedge x
\not= 0) \implies \Phi(z) ].$$
\end{Definition}

The theory $T$ is to the second-order categorical axiomatization of $\Qb$ as first-order $PA$ is to the second-order categorical axiomatization of $\Nb$, and as $RCF$ is to the second order categorical axiomatization of $\mathbb{R}$.  The theory \(T\) is a proper subset of the theory \(Th( \Qb )\), as $Th(\Qb)$ is not recursively axiomatizable. This follows from a result of J. Robinson in \cite{Robinson} that $Th(\Nb)$ is interpretable in $Th(\Qb)$. The models of \(T\) are the fields of characteristic \(0\) which have no proper subfields that are definable with parameters.

The following is related to a result in \cite{HirschfeldModels}, but proven in this form in \cite{DimitrovHarizanovMillerMourad}.

\begin{Proposition} The structure \(\mathbb{N}\) is \(\Pi_{3}^{0}\)-definable in \(\prod_{C}^{}\mathbb{N}\) without parameters.
\end{Proposition} 

The following proposition follows from the work of Koenigsmann in \cite{Koenigsmann}.

\begin{Proposition}[\cite{DimitrovHarizanovMillerMourad}]
$\prod_C \Nb $ is $\Sigma_2^0$-definable in $\prod_C \Qb$.
\end{Proposition}

\begin{Theorem}
     A cohesive power \(\prod_{C}^{}\mathbb{Q}\) has definable proper subfields and is thus not a model of \(T\). 
\end{Theorem}

\begin{proof}Since \(\mathbb{N}\) is \(\Pi_{3}^{0}\)-definable in
\(\prod_{C}^{}\mathbb{N}\), and \(\prod_{C}^{}\mathbb{N}\) is
\(\Sigma_{2}^{0}\)-definable in \(\prod_{C}^{}\mathbb{Q}\), it follows
that \(\mathbb{N}\) is first-order definable in
\(\prod_{C}^{}\mathbb{Q}\), so let \(\Phi(x)\) be a first-order
formula defining \(\mathbb{N}\) in \(\prod_{C}^{}\mathbb{Q}\). Then the formula
$$\exists m\exists n\exists y[\Phi(m)\wedge\Phi(n)\wedge ny = 1\wedge (x = my\ \vee\ x + my = 0)]$$
is a first-order definition of \(\mathbb{Q}\) in
\(\prod_{C}^{}\mathbb{Q}\), so \(\mathbb{Q}\) is a definable proper subfield of \(\prod_{C}^{}\mathbb{Q}\), and thus
\(\prod_{C}^{}\mathbb{Q}\) is not a model of \(T\).
\end{proof}

\section{Cohesive powers of algebraic number fields}\label{alg}

An algebraic number field is a finite extension of \(\mathbb{Q}\).  Since field extensions of characteristic $0$ are always separable, it follows from the primitive element theorem that all number fields are of the form $\Qb (\alpha)$, where $\alpha \in \overline{ \Qb}$ is an algebraic number.  

If \(K\) is a number field, then \(O_{K}\), the ring of integers of \(K\), is the set of all roots in \(K\) of monic polynomials in \(\mathbb{Z}\lbrack x\rbrack\). The following result is due to Park, and is a generalization of an earlier result by Koenigsmann \cite{Koenigsmann}.

\begin{Theorem} \label{park theorem}
    (Park \cite{Park}) If \(K\) is a number field, then \(O_{K}\) is
\(\Pi_{1}^{0}\)-definable in \(K\).
\end{Theorem}

The analog to Hilbert's 10th problem for rings of integers over algebraic number fields was a significant open problem in number theory for decades.  Significant work was done by Shlapentokh in \cite{shlapentokh2007ringsalgebraicnumbersinfinite} and Poonen in \cite{Poonen_2002}.  Finally, the general problem was solved in December 2024 by Koymans and Pagano \cites{KoymansPagano}, in January 2025 by Alpöge, Bhargava, Ho, and Shnidman \cites{ABHS}, and again in May 2025 by Zywina (\cites{zywina2025rankellipticcurvesrank}) .  

\begin{Theorem} \label{H10 for O_K}
    (\cites{KoymansPagano, ABHS, zywina2025rankellipticcurvesrank}) The  ring \(\mathbb{Z}\) is \(\Sigma_{1}^{0}\)-definable in \(O_{K}\) for all number fields \(K\).
\end{Theorem}

\begin{Theorem}
    If \(K\) is a number field, then \(\prod_{C}\mathbb{N}\) is definable in \(\prod_{C}^{}K\).
\end{Theorem}

\begin{proof} 

By Theorem \ref{park theorem}, $O_K$ is $\Pi_1^0$ definable in $K$.  Thus, by Corollary \ref{One Quant Definitions Lift}, $\prod_C O_K$ is definable in $\prod_C K$.  By Theorem \ref{H10 for O_K}, $\Nb$ is $\Sigma_1^0$ definable in $\prod_C O_K$.  Thus, $\prod_C \Nb$ is $\Sigma_1^0$ definable in $\prod_C O_K$.  Chaining definitions, we see that $\prod_C \Nb$ is definable in $\prod_C K$. 

\end{proof}

\begin{Corollary}A cohesive power of a number field \(\prod_{C}K\) is not elementarily equivalent to $K$.
\end{Corollary} 

\begin{proof}
    There is a $\Pi_3^0$ sentence $\Psi$ with $\Nb \models \Psi$ but $\prod_C \Nb \not \models \Psi$.  Since $\Nb$ is definable in $K$, and since $\prod_C \Nb $ is definable in $\prod_C K$ by the same formula, there is a sentence that distinguishes them.
\end{proof}

\begin{Theorem} \label{Coh Pow number fields have same degree}
If $K = \Qb(\alpha)$, then $\prod_C K = (\prod_C \Qb) (\alpha)$.  Furthermore, the minimal polynomial of $\alpha$ over $\prod_C \Qb$ is the minimal polynomial of $\alpha$ over $\Qb$.  Hence, we have $|\prod_C K : \prod_C \Qb | = |K:\Qb|.$ \end{Theorem}

\begin{proof}

The domain $\Qb$ is a computable subset of $K$, since it is c.e.\ and we can computably enumerate $\Qb^c$ by applying every non-constant polynomial in $\Qb[x]$ of degree less than $| K: \Qb|$ to $\alpha$.  

Since $\Qb$ is computable, we may augment $K$ with a predicate $Q$ for membership in $\Qb$ without changing the computability of the structure.  Let $p$ be the minimal polynomial of $\alpha$ over $\Qb$.  Since $K = \Qb(\alpha)$ we have

$$ (K, Q) \models \exists z \forall x \exists y_0\dots  \exists y_{n-1} \left( (x=y_{n-1} z^{n-1} +\cdots+y_0)\wedge p(z)= 0\wedge \bigwedge_{ i < n } Q(y_i )\right),$$ where $n= \deg(\alpha)$. Thus, as this is $\Sigma^0_3$, for the cohesive power we have 

$$\prod_C(K,Q) \models\exists z \forall x \exists y_0\dots  \exists y_{n-1} \left((x=y_{n-1} z^{n-1} +\cdots+y_0)\wedge p(z)= 0 \wedge\bigwedge_{ i < n }Q(y_i) \right). $$  

In the cohesive power, the predicate $Q$ ranges over $\prod_C \Qb $ by Corollary \ref{One Quant Definitions Lift}.  Therefore, $\prod_C K$ is $(\prod_C \Qb) (\alpha)$.

Furthermore, $$(K,Q) \models \forall a_{0}\dots   \forall a_{n-1} \hspace{2mm}\left( \bigwedge_{i < n} Q(a_i)\Imp a_{n-1} \alpha^{n-1} +\cdots+ a_0 \not =0\right).$$

Therefore, since this sentence is $\Pi_1^0$ with a parameter, we have

 $$\prod_C (K,Q) \models \forall a_{0}\cdots \forall a_{n-1} \hspace{2mm} \left(\bigwedge_{i < n} Q(a_i)\Imp a_{n-1} \alpha^{n-1} +\dots  + a_0 \not =0\right).$$

 Thus, $\alpha$ has the same degree in $\prod_C K $ as in $K$, and therefore the same minimal polynomial.  Hence, we have $|\prod_C K : \prod_C \Qb | = |K:\Qb|.$

\end{proof}

\begin{Proposition}
    The algebraic part $\text{ \emph{Alg}}_\Qb (\prod_C K)$ is $K$.
\end{Proposition}

\begin{proof}
    Let $q_0,\dots,q_n \in \Qb$, and let $[\alpha] \in \prod_C K$ such that $\prod_C K \models q_n [\alpha]^n +\dots+q_0=0$.  Then by the fundamental theorem of cohesive powers, $ K \models q_n (\alpha(i))^n + \dots +q_0 = 0$ for almost every $i$.  The field $K$ has at most $n$ many roots of $q_n x^n +\dots +q_0$, call them $\beta_0,\dots,\beta_m$, for $m < n$.  Thus $\alpha(i)$ is among the $\beta_j$'s for almost every $i$.  Since there are finitely many $\beta_j$'s, it follows that $\alpha$ is eventually constant on $C$.  Thus, $\prod_C K \models [\alpha] = \beta_k$ for some $k \leq m$.  Therefore, $[\alpha] \in K$.  Since $[\alpha]$ was arbitrary, $Alg_\Qb (\prod_C K) = K$.
\end{proof}

\begin{Theorem} \label{Finite Galois Groups Cohesive Power Finite Extensions}
If \(K\) is a number field, then
\(Aut\left( \prod_{C}^{}K/\prod_{C}^{}\mathbb{Q} \right)\) is isomorphic
to \(Aut\left( K\mathbb{/Q} \right)\).
\end{Theorem}  

\begin{proof}

Let $K = \Qb(\alpha)$.

First, we may extend $\sigma \in Aut (K/ \Qb)$ to a $\Psi_\sigma \in Gal(\prod_C K / \prod_C \Qb)$, by setting $\Psi_\sigma ([ \beta]) = [\sigma \circ \beta]$.  Since every $\sigma \in Aut(K/\Qb)$ is computable, it follows from a straightforward application of the fundamental theorem of cohesive powers that $\Psi_\sigma$ is a well-defined automorphism.  

Let $\Psi \in Aut (\prod_C K / \prod_C \Qb)$.  Since $\prod_C \Qb$ is definable in $\prod_C K$, and since $\Qb$ is definable in $\prod_C \Qb$, $\Qb$ is definable in $\prod_C K$. Therefore $K$ is definable in $\prod_C K$.  Thus $\Psi$ must send $K$ to $K$, meaning $\sigma_\Psi = \Psi \upharpoonright K $ is an element of $ Aut(K/ \Qb)$.

It suffices to show these two maps are inverse.  This follows immediately, as both are determined by their action on $\alpha$, $\sigma_{\Psi_\sigma} (\alpha) = \Psi_\sigma (\alpha) = \sigma(\alpha)$ and, $\Psi_{\sigma_\Psi} (\alpha) = \sigma_\Psi (\alpha) = \Psi(\alpha) $.

\end{proof}

This answers positively question 6.4 of the 2019 AIM problem list in \cite{AIMDDC}. We also have the following corollaries.

\begin{Corollary}
    Let \(K\) be a number field. Then \(K\) is rigid, if and only if \(\prod_{C}^{}K\) is rigid. Further, 
\(Aut\left( \prod_{C}^{}K \right) = Aut\left( \prod_{C}^{}K/\prod_{C}^{}\mathbb{Q} \right)\).
\end{Corollary}

\begin{Corollary}
    If $K/\Qb$ is Galois, then so is $\prod_C K /\prod_C \Qb$, and thus $Gal(\prod_C K /\prod_ C \Qb)\cong Gal(K/\Qb)$.
\end{Corollary}

\begin{proof}
    By Theorem \ref{Coh Pow number fields have same degree} we have that $|\prod_C K: \prod_C \Qb| = |K:\Qb|$.  Now a finite separable field extension is Galois if and only if its degree is equal to the size of its automorphism group.  Therefore, $$\left|\prod_C K: \prod_C \Qb\right| = |K:\Qb| = |Gal(K/\Qb)| = \left|Aut\left(\prod_C K/\prod_C \Qb\right)\right|$$
    
    Thus $\prod_C K /\prod_C \Qb$ is Galois.
\end{proof}

\section{Cohesive powers of infinite extensions} \label{inf}

There is a class of fields of infinite degree over their prime subfield for which $\Zb$ is low-level definable. 

\begin{Proposition}(\cite{Den3}) If \(F = R(t)\) where \(R\) is a formally real field which contains the algebraic real numbers, then \(\mathbb{Z}\) is
\(\Sigma_{1}^{0}\)-definable in \(F\).
\end{Proposition} 

\begin{Theorem}If \(F = R(t)\) where \(R\) is a formally real field which
contains the algebraic real numbers, then \(\prod_{C}^{}F\) is not
elementarily equivalent to \(F\).
\end{Theorem} 

\begin{proof}Since \(\mathbb{Z}\) is \(\Sigma_{1}^{0}\)-definable in \(F\), it follows that \(\mathbb{N}\) is \(\Sigma_{1}^{0}\)-definable in \(F\). Since this is at a sufficiently low level, by Corollary \ref{One Quant Definitions Lift}, it follows that the same \(\Sigma_{1}^{0}\)
formula that defines \(\mathbb{N}\) in \(F\) also defines
\(\prod_{C}^{}\mathbb{N}\) in \(\prod_{C}^{}F\). Thus since
\(\prod_{C}^{}\mathbb{N}\) is not elementarily equivalent to
\(\mathbb{N}\), it follows that \(\prod_{C}^{}F\) is not elementarily
equivalent to \(F\).
\end{proof}

Now, we study the cohesive products of field extensions in more generality. 

\begin{Proposition}
    Let $\{(K_i, F_i, \Psi_i)\}_{i\in \omega}$ be a uniformly computable sequence of field extensions, and let $C$ be a cohesive set. Then $\prod_C F_i$ is isomorphic to a subfield of $\prod_C K_i$, with natural inclusion map $\theta : \prod_C F_i \hookrightarrow \prod_C K_i$ with $\theta( [\alpha]) = [\varphi]$, where $\varphi(i)= \Psi_i(\alpha(i))$.
\end{Proposition}

\begin{proof}
    Let $[\alpha], [\beta] \in \prod_C F_i$.  Since $\alpha(i)$ is defined for almost every $i$, so is $\Psi_i (\alpha(i))$.  Further, since $\text{ran}(\Psi_i) \subseteq K_i$, $\alpha(i) \in K_i$ for almost every $i$.  Thus, $\theta([\alpha]) \in \prod_C K_i$.  Further $\theta( [\alpha] + [\beta]) = \theta ([\alpha]) + \theta ([\beta])$, as $ \Psi_i(\alpha(i)+\beta(i))= \Psi_i(\alpha(i))+ \Psi_i(\beta(i))$ for almost every $i$.  Similarly, $\theta([\alpha][\beta])=\theta([\alpha])\theta([\beta])$.  Thus, $\theta$ is a ring homomorphism between fields, and is therefore injective.  
\end{proof}

From here on out, we will make no reference to the $\Psi_i$'s within the definition of a computable field extension to avoid clutter.  

We will begin by studying the algebraic/transcendental behavior of cohesive products of field extensions.

\begin{Lemma} \label{classification of algebraic elements of cohprods}

Let $\{K_i/F_i \}_{i \in \omega}$ be a uniformly computable sequence of field extensions.  Let $[\alpha] \in \prod_C K_i $.  $[\alpha]$ is algebraic over $\prod_C F_i$ with degree $n$ if and only if  $\deg_{F_i} (\alpha(i)) = n$ for almost all $i \in \omega$.

\end{Lemma}

\begin{proof}

Assume $[\alpha]$ is algebraic with degree $n$.  Thus, $ [f_n] [\alpha]^n +\cdots+ [f_0] = 0  $ for some $[f_0],\dots  ,[f_n] \in \prod_C K$.  Thus, $f_n(i) \alpha(i)^n+\cdots+ f_0(i) = 0 $ for almost all $i$, by the fundamental theorem of cohesive products.

Assume now that, for almost all $i$, $\deg_{F_i} (\alpha(i)) = n$.  Let $f_n(i),\dots  ,f_0(i)$ be the coefficients of the first polynomial  in  $F[x]$ with $f_n (i) \alpha(i)^n +\cdots+f_0(i) = 0$. Function $f_j$ is partial computable, and halts on almost all $i$.  Thus, $[f_n] [\alpha]^n +\cdots+ [f_0] = 0$.
    
\end{proof}
As an easy corollary, we have the following.

\begin{Corollary} \label{Corollary for fixing that one proof}
    Let $K/F$ be a computable field extension, and let $[f_0],[f_1],\dots  ,[f_n] \in \prod_C K$. Element $[f_0]$ is algebraic over $\prod_C F([f_1],\dots  ,[f_n])$ with degree $m$ if and only if $ [F(f_0(i),\dots  ,f_n(i)) :F( f_1(i),\dots  ,f_n(i))] =m$ for almost every $i$.  
\end{Corollary}

\begin{Proposition}
    Let $K_i / F_i$ be a uniformly computable sequence of computable field extensions.  Assume further that there is a single $n \in \Nb$ so that for every $i$ and every $\alpha \in K_i$, $\deg_{F_i} (\alpha) \leq n$.  Then $\prod_C K_i / \prod_C F_i$ is algebraic.
\end{Proposition}

\begin{proof}

Apply Lemma \ref{classification of algebraic elements of cohprods} to every element of $\prod_C K_i$.  

\end{proof}

\begin{Proposition} 
    Let $K/F$ be a computable transcendental extension.  Then $\prod_C K / \prod_C F$ has infinite transcendence degree.  
\end{Proposition}

    \begin{proof}
        Let $ t \in K$ be transcendental over $F$.  Consider the functions $f_n (i) = t^{i^n}$.  Let $k_1<\dots  <k_n \in \omega$. Then $[F(f_{k_1}(i),\dots  .,f_{k_n}(i) ): F(f_{k_2}(i),\dots  ,f_{k_n}(i))]=[F(t^{i^{k_1}},\dots  ,t^{i^{k_n}}):F(t^{i^{k_2}},\dots  ,t^{i^{k_n}})]= [F(t^{i^{k_1}} ):F(t^{i^{k_2}})]$ as $t^{i^{k_1}}\mid t^{i^{k_2}} \mid \dots  \mid t^{i^{k_n}}$.  Further, $[F(t^{i^{k_1}} ):F(t^{i^{k_2}})]= i^{k_2 -k_1} \geq i$, as $t^{i^{k_1}}$ has minimal polynomial $x^{i^{k_2-k_1}} - t^{i^{k_2}}$, where minimality follows from $t$ being transcendental over $F$.  Thus, by Corollary \ref{Corollary for fixing that one proof}, $[f_{k_1}]$ is transcendental over $(\prod_C F)([f_{k_2}],\dots  ,[f_{k_n}])$, as $[F(f_{k_1}(i),\dots  .,f_{k_n}(i) ): F(f_{k_2}(i),\dots  ,f_{k_n}(i))] $ is not equal to any fixed natural number for almost every $i$.
    \end{proof}

\begin{Theorem}
    Let $F \hookrightarrow K$ be a computable field extension with $F$ having a computable degree map $\deg:K\to \Nb $. Further, assume that $\{\deg_F (\alpha) \mid \alpha \in K\} $ is unbounded. Then $\prod_C K / \prod_C F$ is transcendental with infinite transcendence degree. 
\end{Theorem}

\begin{proof}

We define the $f_n$'s via recursion.  Let $f_0(0) = 0$.  Assuming $f_0(i)$ is defined, let $f_0(i+1)$ be the first encountered element of $K$ with $\deg_F (f_0 (i) ) < \deg_F (f_0 (i+1)) $.  Clearly $f_0$ is total computable. 

Assume $f_0,\dots , f_n$ have already been defined.  Let $f_{n+1}(0) = 0$, and let $f_{n+1} (i
+1)$ be the first element of $K$ encountered with 

$$\deg_F (f_{n+1} (i+1)) > \deg(f_{n+1}(i)) \prod_{j\leq n} \deg(f_j(i+1)).$$

Clearly, each $f_n$ is total computable. We show that $[ F(f_0(i),\dots  ,f_{n+1}(i)) : F(f_0(i),\dots  ,f_n(i))]$ is strictly increasing in $i$. Certainly, $ [ F(f_0(i+1),\dots  ,f_n(i+1)) : F ] \leq \prod_{j\leq n} \deg_F ( f_j(i+1))$.  By construction, $[ F(f_{n+1}(i) ) : F] > \deg (f_{n+1} (i)) \prod_{j\leq n} \deg(f_j(i+1))  ]$.  Therefore,
$$ [F(f_0(i+1),\dots  ,f_{n+1}(i+1) ) : F(f_0(i+1),\dots  ,f_n(i+1))] = \frac{[F(f_0(i+1),\dots  ,f_{n+1}(i+1) ) : F]}{[F(f_0(i+1),\dots  ,f_{n}(i+1) ) : F]} $$
$$\geq \frac{\deg_F (f_{n+1}(i+1))} {\prod_{j\leq n} \deg_F (f_j(i+1))}> \frac{\deg (f_{n+1}(i))\prod_{j\leq n} \deg_F (f_j(i+1)))}{\prod_{j\leq n} \deg_F (f_j(i+1))} $$

$$= \deg_F (f_{n+1}(i)) \geq [ F(f_0(i),\dots  ,f_{n+1}(i) ): F(f_0(i),\dots  ,f_n(i)) ]$$

In other words,

% https://q.uiver.app/#q=WzAsNCxbMiw0LCJGIl0sWzAsMiwiRihmX3tuKzF9KHgrMSkpIl0sWzQsMiwiRihmXzAoeCsxKSwuLi4sZl9uKHgrMSkpIl0sWzIsMCwiRihmXzAoeCsxKSwuLi4sZl9uKHgrMSksZl97bisxfSh4KyEpKSJdLFswLDEsIj4gXFxkZWcoZl97bisxfSh4KSlcXHByb2Rfe2lcXGxlcSBufSBcXGRlZyhmX2kgKHgrMSkpIiwwLHsic3R5bGUiOnsiaGVhZCI6eyJuYW1lIjoibm9uZSJ9fX1dLFswLDIsIlxcbGVxIFxccHJvZF97aVxcbGVxIG59XFxkZWcoZl9pKHgrMSkpIiwyLHsic3R5bGUiOnsiaGVhZCI6eyJuYW1lIjoibm9uZSJ9fX1dLFsxLDMsIiIsMCx7InN0eWxlIjp7ImhlYWQiOnsibmFtZSI6Im5vbmUifX19XSxbMywyLCI+IFxcZGVnKGZfe24rMX0oeCkpIiwwLHsic3R5bGUiOnsiaGVhZCI6eyJuYW1lIjoibm9uZSJ9fX1dXQ==
\[\begin{tikzcd}
	&& {F(f_0(i+1),\dots  ,f_n(i+1),f_{n+1}(i+1))} \\
	\\
	{F(f_{n+1}(i+1))} &&&& {F(f_0(i+1),\dots  ,f_n(i+1))} \\
	\\
	&& F
	\arrow["{> \deg(f_{n+1}(i))}", no head, from=1-3, to=3-5]
	\arrow[no head, from=3-1, to=1-3]
	\arrow["{> \deg(f_{n+1}(i))\prod_{j\leq n} \deg(f_j (i+1))}", no head, from=5-3, to=3-1]
	\arrow["{\leq \prod_{j\leq n}\deg(f_j(i+1))}"', no head, from=5-3, to=3-5]
\end{tikzcd}\]

Observe that for each $k$, there are only finitely many $i$ with $[ F(f_0 (i),\dots  ,f_{n+1}(i)):F(f_0(i),\dots  ,f_{n}(i)] \leq k$.  Thus, by Corollary \ref{Corollary for fixing that one proof}, $[f_{n+1}]$ is transcendental over $(\prod_C F )([f_0],...,[f_n])$.

\end{proof}

\begin{Corollary}
    Let $K/F$ be a computable algebraic extension, with $F$ having a computable degree map $\deg: K \to \Nb$. Then  $(\prod_C K)/(\prod_C F)$ is algebraic if and only if $\{ \deg_F (\alpha) \mid \alpha \in K\}$ is bounded. 
\end{Corollary}

Note that this corollary applies to every field extension $K/F$ where $F$ has a splitting algorithm.  In light of the above corollary, it is natural to ask if the requirement that $F$ have a splitting algorithm is sharp.  

\begin{Question}
    Does there exists a computable field extension $K/F$ with $\{ \deg_F (\alpha) \mid \alpha \in K \}$ unbounded, but $\prod_C K/ \prod_C F$ algebraic?
\end{Question}

The following proposition demonstrates that such a field extension must necessarily be non-separable.  In particular, the fields must be of finite characteristic.

\begin{Proposition}
    
    Let $K/F$ be a separable computable field extension with $\{ \deg_ F (\alpha)\mid \alpha \in K \}$ unbounded.  Then $\prod_C K / \prod_C F $ is transcendental. 
    
\end{Proposition}

\begin{proof}
    Let $\alpha_0,\alpha_1,\dots  $ be an effective enumeration of $K$. We define $\beta$ via recursion.  Set  $\beta(0) = \alpha_0$, and take $\beta(i+1) $ to be the first enumerated element of $K$ with $\beta(i) , \alpha_i \in F(\beta(i+1) )$.   By the primitive element theorem, such a $\beta(i+1)$ is guaranteed to exist.

    Now $\deg_F ( \beta(i)) \leq \deg_F(\beta(i+1))$, and $\deg_F (\beta(i))$ grows arbitrarily large, as every $\gamma \in K$ is in all but finitely many $F(\beta(i))$.   Thus, $[\beta] \in \prod_C K$ is transcendental over $\prod_C F$, by Lemma \ref{classification of algebraic elements of cohprods}.
    
\end{proof}

This proof gives no access to the transcendence degree of $\prod_C K /\prod_C F$, which raises the following question:

\begin{Question}
    Does there exists a computable separable field extension $K/F$, with $\prod_C K/ \prod_CF$ having transcendence degree exactly 1?
\end{Question}

The two questions seem intuitively equivalent.  That is, whatever technique is used to answer one question will likely be able to resolve the other.

\subsection{Cohesive powers of infinite Galois extensions}

In this section, we give a description of the Galois group of the algebraic parts of cohesive powers of infinite Galois extensions in terms of the finite Galois groups of the base fields' finite Galois sub-extensions.  

\begin{Theorem}
    Let $\{K_i/F\}_{i\in \omega}$ be a uniformly computable sequence of Galois extension. Then $\text{Alg}( \prod_C K_i) / \prod_C F$ is Galois.  
\end{Theorem}

\begin{proof}

 Let $[\alpha] \in 
\text{Alg} \left( \prod_C K \right) \backslash \left( \prod_C F \right)$ and let $[a_n] x^n +\cdots+ [a_0]$ be its minimal polynomial.  

Clearly,

$$\prod_C K_i \models \left( [a_n] [\alpha]^n +\cdots+[a_0] =0 \right) \wedge \forall b_{n-1}\dots  \forall b_0 \left( b_{n-1} [\alpha]^{n-1}+\cdots+b_0 \not=0 \right) $$ Thus, for almost every $i$, 

$$K_i \models \left( a_n(i) \alpha(i)^n +\cdots+a_0(i) =0 \right) \wedge \forall b_{n-1}\dots  \forall b_0 \left( b_{n-1} \alpha(i)^{n-1}+\cdots+b_0 \not=0 \right)$$ We conclude that for almost every $i$, $a_n(i) x^n+\cdots+a_0(i)$ is the minimal polynomial of $\alpha(i)$ over $\prod_C F$.  Since $K_i/F$ is Galois, $a_n(i) x^n +\cdots+a_0(i)$ has $n$ distinct roots.

Define $\beta_1$,\dots,$\beta_n$ by recursion.  Set $\beta_1 = \alpha$.  Assuming $\beta_m$ has already been defined for $m<n$, we define $\beta_{m+1}(i)$ by first computing $a_n(i),\dots  ,a_0(i)$, and searching through $K_i$ for roots of $a_n(i) x^n+\cdots+a_0(i)$. We set $\beta_{m+1}(i)$ equal to the first such root we find that is distinct from $ \beta_1(i),\dots ,\beta_m(i)$.  It is clear that $\beta_m$ is defined for almost every $i$.  

Therefore, we have that for almost every $i$,

$$ K_i \models \forall x \left(a_n(i) x^n+\cdots+a_0(i) = (x-\beta_1(i))\cdots(x-\beta_n(i))\right) \wedge \bigwedge_{m<\ell\leq n} \beta_m(i) \not= \beta_\ell (i)$$ Thus, 

$$\prod_C K_i  \models \forall x\left( [a_n] x^n+\cdots+[a_0] = (x-[\beta_1])\dots  (x-[\beta_n])\right) \wedge \bigwedge_{m<\ell\leq n} [\beta_m] \not= [\beta_\ell]$$

We thus conclude $[a_n]x^n+\cdots+[a_0]$ splits completely over $\prod_C F$ into unique linear factors.  Thus, $\text{Alg}(\prod_C K_i) /\prod_C F$ is Galois.  
\end{proof}
 
 \begin{Theorem}\label{Bad Description of the Kernel}
     Let $\{K_i/ F\}_{i\in \omega}$ be a uniformly computable sequence of finite Galois extensions with uniformly computable Galois groups $G_i$.  Assume, furthermore, that the group actions $G_i \curvearrowright K_i$ are uniformly computable.  Then $\prod_C G_i$ acts faithfully on $\prod_C K_i$, and is therefore isomorphic to a subgroup of $\text{Aut} (\prod_C K_i /\prod_C F )$.

    Further, the map $\prod_C G_i \to Gal ( \text{Alg}(\prod_C K_i )/\prod_CF )$, given by restricting the action of $[\sigma]\in \prod_C G_i$ onto the algebraic part, $ [\sigma]\upharpoonright \text{Alg}(\prod_C K_i)$, has dense image. Moreover, the kernel is exactly 
    $$\left\{ [\sigma] \middle| \forall n   \left\{i \middle| \sigma (i) \in \bigcap_{H\leq G ,[G:H] = n }H  \right\} \supseteq^* C \right\} $$
    
 \end{Theorem}
\begin{proof}
Let $[\sigma] \in \prod_C G_i$, and let $[\alpha]\in \prod_C K_i$.  Set $[\sigma][\alpha] = [\sigma \circ \alpha]$.  This is a well-defined group action by the fundamental theorem of cohesive products.  This gives us a well-defined map $ \prod_C G_i \to \text{Aut} (\prod_C K_i)$.  For every  $[\sigma] \in \prod_C G_i$ and $[\alpha] \in \prod_C F$, we have $[\sigma][\alpha] = [\sigma \circ \alpha] = [\alpha]$, as $\sigma(i) \in Gal(K_i/F)$ for almost every $i$. Therefore, the image of this map lands inside $\text{Aut} (\prod_C K_i /\prod_C F)$.    

Suppose $[\sigma] \in \prod_C G_i$ is not the identity. Let $\alpha(i)$ be the first encountered element of $K_i$ that is not fixed by $\sigma(i)$. Then $\alpha$ is a partial computable function that is defined almost everywhere, as $\sigma(i) \not= e_{G_i}$ for almost every $i$.  By construction, $[\sigma][\alpha] \not = [\alpha]$.  Thus, no nontrivial element of $\prod_C G_i$ acts trivially on $\prod_C K_i$.  Therefore, the action $\prod_C G_i \curvearrowright \prod_C K_i$ is faithful.  That is, the map $\prod_C G_i \to \text{Aut} (\prod_C K_i)$ is injective.

Now, we turn our attention to the restriction of the map to the algebraic part. To show that the image of $\prod_C G_i$ is dense in $Gal ( \text{Alg}(\prod_C K_i )/ \prod_C F)$, it is sufficient to show that $E^{\prod_C G_i} = F$ by Lemma \ref{same fixed field iff same closure}.  That is, for every $[\alpha] \in \text{Alg}(\prod_C K_i) - \prod_C F$, there is some $[\sigma] \in \prod_C G_i$ with $[\sigma][\alpha] \not= [\alpha]$.

Let $\alpha \in \text{Alg}(\prod_C K_i) - \prod_C F$.  Let $\sigma(i)$ be the first encountered element of $G_i$ that doesn't fix $\alpha(i)$.  $\sigma$ is a partial computable function defined almost everywhere.  By construction, $[\sigma][\alpha] \not= [\alpha]$.  Thus, $\prod_C G_i $ has fixed field $\prod_C F$, and therefore has dense image in $Gal(\text{Alg}(\prod_C K_i) /\prod_C F)$.

Now, we turn our attention to the kernel.  

First, let $ [\tau] \in \{ [\sigma]\mid \forall n   \{i\mid \sigma (i) \in \bigcap_{H\leq G, [G:H] =n }H  \} \supseteq^* C \} $, and let $ [\alpha] \in \text{Alg}(\prod_C K_i)$.  By Lemma \ref{classification of algebraic elements of cohprods}, $\alpha(i)$ has degree $n$ for almost every $i$.  Let $H_i = Gal(K/F(\alpha(i)))$.  Thus, for almost every $i$, $[G_i : H_i ]$ has degree $n$.  Thus, for almost every $i$, $\tau(i) \in H_i$ as $\tau (i) \in \bigcap_{H\leq G, [G:H] = n }H$ for almost all $i$. Thus, $\tau(i)$ acts trivially on $\alpha(i)$, for almost every $i$.  That is, $[\sigma][\alpha] = [\alpha].$ Since $[\tau]$ and $[\alpha]$ were arbitrary, we conclude that $\{ [\sigma]\mid \forall n  \{i\mid \sigma (i) \in \bigcap_{H\leq G, [G:H] = n }H  \} \supseteq^* C \}  $ is contained in the kernel of the map $\prod_C G_i \to   Gal(\text{Alg} (\prod_C K_i) /\prod_C F )$.

Now suppose $[\tau] \not\in  \{ [\sigma]\mid \forall n\{i\mid \sigma (i) \in \bigcap_{H\leq G, |G:H| = n }H  \} \supseteq^* C \}$. Thus there is some $n$ such that for almost every $i$, there is some $H_i \leq G_i$ with $|G_i : H_i| =n$, and $ \tau(i) \not \in H_i$.  As each of the $G_i$'s is finite, $H_i$ may be found computably.  We define $\alpha(i)$ by searching through $K_i$ for some $\beta$ with $\tau(i) \beta \not = \beta$, and $\rho \beta = \beta$ for every $\rho \in H_i$, i.e., $\beta \in E^{H_i}$.  Such a $\beta$ must exist by the fundamental theorem of Galois theory, and detecting it is effective.  Take $\alpha(i)$ to be the first encountered $\beta$.  

Since, for almost every $i$, $\alpha(i) \in E^{H_i}$ and $[E^{H_i} : F ] = | G_i: H_i | = n$, we must have that $ \deg (\alpha (i)) \leq n $ for almost every $i$.  Thus, by Lemma \ref{classification of algebraic elements of cohprods}, $[\alpha] \in \text{Alg}(\prod_C K_i) $.  By construction $[\tau] [\alpha] \not= [\alpha]$.  Thus, $\sigma$ is not in the kernel of the map $ \prod_C G_i \to Gal(\text{Alg}(\prod_C K_i )/ \prod_C F)$.  
     
\end{proof}

A nicer description of the above phenomenon is given at the end of Section 5.  

We would like to be able to decompose arbitrary cohesive powers of infinite Galois extensions into cohesive products of finite sub extensions.  The following definition and lemma give us a canonical way to do this.

\begin{Definition}

Let $\{\mathcal{A}_i \}_{i\in \omega}$ be a uniformly computable sequence of structures, and let $\varphi : \omega \to \omega$ be a partial computable function with $\text{dom} (\varphi) \supseteq^* C$.  

$$\prod_{C, \varphi} \mathcal{A}_i := \prod_C \mathcal{A}_{\varphi (i) }$$
    
\end{Definition}

\begin{Lemma} 
    Let $\varphi_1$, $\varphi_2$ be partial computable functions defined almost everywhere.  If $\varphi_1 \simeq_C \varphi_2$, then $\prod_{C, \varphi_1} \mathcal{A}_i \cong \prod_{C,\varphi_2} \mathcal{A}_i$.
\end{Lemma}
\begin{proof}
    Let $\alpha \in \prod_{C, \varphi_1} \mathcal{A}_i$.  Let $\alpha^* (i) = \begin{cases} \alpha(i) & \text{if } \varphi_1(i)\downarrow =\varphi_2(i)\downarrow \\ \uparrow & \text{otherwise} \end{cases}$.  
    
    Clearly, $[\alpha^*] \in \prod_{C, \varphi_2} \mathcal{A}_i$ .  The map $[\alpha]\mapsto [\alpha^*]$ is the desired isomorphism.   
\end{proof}

Since $\varphi_1\simeq_C \varphi_2 $ if and only if $[\varphi_1]=[\varphi_2]$ in $\prod_C \Nb$, there is a natural extension of the sequence  $\{ \prod_C \mathcal{A}_i \}_{i\in \Nb}$ to $\{ \prod_{C,[\varphi] } \mathcal{A}_i \}_{[\varphi] \in \prod_C \Nb }$.  We will use capital letters $N,N_1,N_2,\dots  M,M_1,M_2,\dots  $ to represent elements of $\prod_C \Nb$, which we refer to as hyper-natural numbers. Thus, we have the sequence $\{ \prod_{C, M} \mathcal{A}_i \}_{M\in \prod_C\Nb}$. 

\begin{Observation} \label{Not so clear after all}
    If $\mathcal{A}_i \subseteq \mathcal{A}_{i+1}$ for all $i$, then $\prod_{C,N} \mathcal{A}_i \subseteq \prod_{C,M} \mathcal{A}_i$ for all $N \leq M$ in $\prod_C \Nb$.
\end{Observation}

From here on out, brackets will be omitted from elements of cohesive products to avoid notational clutter.  That is, $[\alpha] \in \prod_C K$ will be referred to as $\alpha$.  The variable $i$ will be used exclusively as the argument to the computable functions within the domain of the cohesive power to avoid ambiguity.  Since all reasoning about $\alpha(i)$ will hold for almost all $i$, well-definedness is not obstructed.

\begin{Lemma} \label{sequence-extension-lemma}
    Let $K/F$ be a computable field extension, and let $\{K_i/F\}_{i\in \omega}$ be a uniformly computable sequence of sub-extensions with $\{K_{i+1}/K_i\}_{i\in \omega}$ and $\{K/K_i\}_{i\in \omega}$ being uniformly computable, and with $K = \bigcup_{i\in \Nb} K_i$.  Then $\prod_C K = \bigcup_{M \in \prod_{C} \Nb} \left(\prod_{C, M}K_i\right) $.  
\end{Lemma}

\begin{proof}
    By Observation \ref{Not so clear after all}, each $\prod_{C,M} K_i \subseteq \prod_C K$.   

    Let $\alpha \in \prod_C{K}$.  We compute $M(i)$ by first computing $\alpha(i)$, and then enumerating all $K_i$ simultaneously, until $\alpha(i)$ is enumerated. We set $M(i)=j$, where $K_j$ is the field in which $\alpha(i)$ is enumerated  first.  Thus, we have $\alpha \in \prod_{C, M } K_i$.  
\end{proof}

We can also extend inverse systems of finite groups indexed over $\Nb$ to inverse systems of pseudofinite groups indexed by hyper-naturals in $\prod_C \Nb$.  Since topology and non-standard analysis famously struggle to work together, we will omit topology from the following lemma by giving every group the discrete topology.

\begin{Lemma} \label{New Lemma About Extending Inverse Systems}
    Let $\{ G_i \}_{i\in \omega}$ be a uniformly computable sequence of groups, and for $n< m \in \Nb$,  and let $f_{m,n}: G_{m} \to G_n$ be uniformly computable surjective group homomorphisms, such that for every $m,n,k \in \Nb$ with $m>n>k$ we have $f_{n,k}\circ f_{m,n}= f_{m,k}$.  Then, for every $M,N,K \in \prod_C \Nb $ with $M>N>K$, there are surjective group homomorphisms $f_{M,N} : \prod_{C,M} G_i \to \prod_{C,N}G_i$ such that $f_{N,K}\circ f_{M,N} = f_{M,K}$.
\end{Lemma}

\begin{proof}
    For $g \in \prod_{C,M} G_i$, let $f_{M,N} (g)(i) = f_{M(i),N(i)}(g(i))$.  Clearly, this is a well-defined function $f_{M,N}: \prod_{C,M} G_i \to \prod_{C,N}G_i$.  Let $g,h \in \prod_{C,M} G_i$.  For almost every $i$, we have $f_{M,N}(gh)(i)=f_{M(i),N(i)}(g(i)h(i)) = f_{M(i),N(i)}(g(i))f_{M(i),N(i)}(h(i))= f_{M,N}(g)(i)f_{M,N}(h)(i)$. Thus, $f_{M,N}$ is a group homomorphism.  

    For surjectivity, let $g \in \prod_{C,N}G_i$.  We compute $h(i)$ by setting it equal to the first encountered element of $G_{M(i)}$ with $f_{M(i),N(i)}(h(i) )=g(i)$.  Such an $h(i)$ always exists whenever $M(i)$, $N(i)$ and $g(i)$ do, as all $f_{M(i),N(i)}$ are surjective, and the search is algorithmic since the $f$'s are uniformly computable.  
\end{proof}

\begin{Definition}\label{def-fully-computably}
    We will call a Galois extension $K/F$ \emph{fully computable} if 

    \begin{enumerate}
        \item $(K,F,\Psi)$ is a computable Galois extension.
        \item There exist uniformly computable finite Galois sub-extensions $K_i\subseteq K$ with $(K,K_i,\Psi_i)$ uniformly computable, $(K_{i+1},K_i,\Phi_i)$ uniformly computable and $K= \bigcup_{i\in \Nb} K_i$.
        \item The groups $G_i =Gal (K_i/F)$ are uniformly computable. 
        \item The quotient maps $f_{m,n}: G_{m} \to G_n$ are uniformly computable, where $n<m$.
        \item The group actions $G_i \curvearrowright K_i$ are uniformly computable.  
        \item $K_i \subsetneq K_{i+1} $.
    \end{enumerate}
\end{Definition}

The reader should note that the assumption that $K_i \subsetneq K_{i+1}$ is deducible  from the others.  We may check if $K_i = K_{i+1}$ by checking if $G_i = G_{i+1}$, in which case, we may skip $K_{i+1}$ in our list.

While this condition might seem too strong on first glance, such fields are actually abundant.  

As an example, take $F=\Qb$ and take $K = \Qb^{ab} = \Qb(\zeta_1,\zeta_2,\zeta_3,\dots  )$, where $\zeta_n$ is a primitive $n$-th root of unity.  Let $K_i = \Qb(\zeta_{i!})$.  Then $G_i = Gal(K_i/\Qb)= \left( \Zb /i! \Zb \right)^\times$ are uniformly computable.  Further, the quotient maps $G_{i+1} \to G_i$ sending $[n]$ to $[n] \text{ Mod } i!$ are uniformly computable.  The group actions $G_i \curvearrowright K_i$ are given by $[n] (\zeta_{i!} )= \zeta_{i!}^n$ and are uniformly computable.

Other examples of field extensions with such a decomposition include $\overline{\Qb}/\Qb$ and $\mathbb{F}_p(t)^{sep}/ \mathbb{F}_p(t)$.  More generally, if $F$ is any field with a splitting algorithm, and $K$ is a c.e.\ Galois subfield of $F^{sep}$, the separable closure of $F$, then $K/F$ admits a fully computable extension.

\begin{Theorem} \label{Thm-Gal-Extension-Commutes}
    Let $K/F$ be a fully computable Galois extension. Let $f_{M,N} : \prod_{C,M} G_i \to \prod_{C,N} G_i$ be the induced maps as in Lemma \ref{New Lemma About Extending Inverse Systems}. The induced maps $f_{M,N}$ commute with the group actions $\prod_{C,N}G_i \curvearrowright \prod_{C,N} K_i $ from Theorem \ref{Bad Description of the Kernel}.  

    i.e.

 % https://q.uiver.app/#q=WzAsOCxbNCwyLCJcXHByb2Rfe0MsIE59IEdfaSJdLFs0LDAsIlxccHJvZF97QyxNfSBHX2kiXSxbNiwyLCJHYWwgXFxsZWZ0KCBBbGcgXFxsZWZ0KCBcXHByb2Rfe0MsTn0gS19pIFxccmlnaHQpIC8gXFxsZWZ0KCBcXHByb2RfQyBGIFxccmlnaHQpIFxccmlnaHQpIl0sWzYsMCwiR2FsIFxcbGVmdCggQWxnXFxsZWZ0KCBcXHByb2Rfe0MsTX0gS19pIFxccmlnaHQpLyBcXGxlZnQoXFxwcm9kX0MgRiBcXHJpZ2h0KVxccmlnaHQpIl0sWzIsMCwiS2VyIChcXGV0YV97TX0pIl0sWzIsMiwiS2VyIChcXGV0YV97Tn0pIl0sWzAsMCwiMCJdLFswLDIsIjAiXSxbMywyXSxbMSwzLCJcXGV0YV9NIl0sWzAsMiwiXFxldGFfTiJdLFs0LDEsIiIsMCx7InN0eWxlIjp7InRhaWwiOnsibmFtZSI6Imhvb2siLCJzaWRlIjoidG9wIn19fV0sWzUsMCwiIiwwLHsic3R5bGUiOnsidGFpbCI6eyJuYW1lIjoiaG9vayIsInNpZGUiOiJ0b3AifX19XSxbNiw0XSxbNyw1XSxbMSwwLCJmX3tNLE59Il0sWzQsNV1d
\[\begin{tikzcd}
	0 && {Ker (\eta_{M})} && {\prod_{C,M} G_i} && {Gal \left( \text{Alg}\left( \prod_{C,M} K_i \right)/ \left(\prod_C F \right)\right)} \\
	\\
	0 && {Ker (\eta_{N})} && {\prod_{C, N} G_i} && {Gal \left( \text{Alg} \left( \prod_{C,N} K_i \right) / \left( \prod_C F \right) \right)}
	\arrow[from=1-1, to=1-3]
	\arrow[hook, from=1-3, to=1-5]
	\arrow[from=1-3, to=3-3]
	\arrow["{\eta_M}", from=1-5, to=1-7]
	\arrow["{f_{M,N}}", from=1-5, to=3-5]
	\arrow[from=1-7, to=3-7]
	\arrow[from=3-1, to=3-3]
	\arrow[hook, from=3-3, to=3-5]
	\arrow["{\eta_N}", from=3-5, to=3-7]
\end{tikzcd}\]
with $\eta_N$ as the induced map from the actions, the rows left exact, and the kernels as in theorem \ref{Bad Description of the Kernel}.

\end{Theorem}

\begin{proof}

We extend our sequence of extensions $\{K_i/F\}_{i\in \omega}$ to $
\left\{ \left( \prod_{C, M} K_i \right)/ \left(\prod_C F \right)\right\}_{M\in \prod_C \mathbb{N}}$ via Lemma \ref{sequence-extension-lemma}.  Via Lemma \ref{New Lemma About Extending Inverse Systems}, we extend the inverse system of Galois groups $G_{i+1} \to G_i$ to a sequence $\prod_{C,M+1} G_i \to \prod_{C,M} G_i$.  

Assume now that $M,N \in \prod_C \mathbb{N}$ are such that $M > N$.  By Theorem \ref{Bad Description of the Kernel}, we have the horizontal maps $\prod_{C,M_i}G_i \rightarrow Gal(  \text{Alg} \left(\prod_{C,M_i} K_i \right) / \prod_{C} F )$.  To show now that the diagrams commute, it is sufficient to show that the maps $f_{M,N}:\prod_{C,M} G_i \to \prod_{C,N} G_i$ agree on the Galois action on $\prod_{C,N} K_i$.  That is, given $\sigma \in \prod_{C, M} G_i$ and $\alpha \in \prod_{C,N} K_i$, we have $\sigma \alpha = f_{M,N} (\sigma) \alpha$.  

    To show the desired equality, it is sufficient to show that the set $\{ i \mid \sigma(i) \alpha(i) = f_{M,N}(i) (\sigma (i)) \alpha(i) \} $ that is equal to the set $ \{ i \mid \sigma(i) \alpha(i) = f_{M(i),N(i)} (\sigma(i))\alpha(i) \}$ contains almost all $i$.  In the case when $\alpha(i) \in K_{N (i)}$ and $\sigma(i)\in  G_{M(i)}$, it follows by the fundamental theorem of Galois theory that $\sigma(i) \alpha(i) = f_{M(i),N(i)}(\sigma(i))\alpha(i)$.  Since $\alpha \in \prod_{C,N}K_{i}$ and $\sigma \in \prod_{C, M} G_i$, we have that $\{ i \mid \alpha(i)\in K_{N (i)} \}\cap \{ i \mid \sigma(i)\in  G_{M(i)} \}\supseteq^* C$.  Thus, $\{ i \mid \sigma(i) \alpha(i) = f_{M(i),N(i)}(\sigma(i))\alpha(i) \} \supseteq \{ i \mid \alpha(i) \in K_{N (i)} \wedge \sigma(i)\in  G_{M(i)} \}=  \{ i \mid \alpha(i)\in K_{N (i)} \}\cap \{ i \mid \sigma(i)\in  G_{M(i)} \} \supseteq^* C$.

\end{proof}

Recall that, by Lemma \ref{Infinite Galois Theory infinite proj lim}

   $$ \varprojlim_{M \in \prod_C \Nb} Gal \left(\text{Alg} \left( \prod_{C,M} K_i \right) / \prod_C F\right) \cong Gal \left(\prod_C K / \prod_C F \right)$$ as topological groups.

\begin{Theorem} \label{Cohesive Product of Actions Commutes with Quotients}
    Let $K/F$ be a fully computable infinite Galois extension with $\prod_C K /\prod_C F$ Galois.  The induced map $\varprojlim _ { N\in \prod_C \Nb} (\prod_{C,N} G_i ) \to \varprojlim  _ { N\in \prod_C \Nb}Gal (\text{Alg}(\prod_{C,N} K_i ) /\prod_C F )$ has dense image.  Moreover, the kernel is given by $\varprojlim _ { N\in \prod_C \Nb} ker(\eta_N)$.
\end{Theorem}

The existence of the inverse system of left exact sequences is enough to prove the theorem, as projective limits are left exact.

% https://q.uiver.app/#q=WzAsOCxbNCwyLCJcXHByb2Rfe0MsIE59IEdfaSJdLFs0LDAsIlxccHJvZF97QyxNfSBHX2kiXSxbNiwyLCJHYWwgXFxsZWZ0KCBBbGcgXFxsZWZ0KCBcXHByb2Rfe0MsTn0gS19pIFxccmlnaHQpIC8gXFxsZWZ0KCBcXHByb2RfQyBGIFxccmlnaHQpIFxccmlnaHQpIl0sWzYsMCwiR2FsIFxcbGVmdCggQWxnXFxsZWZ0KCBcXHByb2Rfe0MsTX0gS19pIFxccmlnaHQpLyBcXGxlZnQoXFxwcm9kX0MgRiBcXHJpZ2h0KVxccmlnaHQpIl0sWzIsMCwiS2VyIChcXGV0YV97TX0pIl0sWzIsMiwiS2VyIChcXGV0YV97Tn0pIl0sWzAsMCwiMCJdLFswLDIsIjAiXSxbMywyXSxbMSwzLCJcXGV0YV9NIl0sWzAsMiwiXFxldGFfTiJdLFs0LDEsIiIsMCx7InN0eWxlIjp7InRhaWwiOnsibmFtZSI6Imhvb2siLCJzaWRlIjoidG9wIn19fV0sWzUsMCwiIiwwLHsic3R5bGUiOnsidGFpbCI6eyJuYW1lIjoiaG9vayIsInNpZGUiOiJ0b3AifX19XSxbNiw0XSxbNyw1XSxbMSwwLCJmX3tNLE59Il0sWzQsNV1d
\[\begin{tikzcd}
	0 && {Ker (\eta_{M})} && {\prod_{C,M} G_i} && {Gal \left( \text{Alg}\left( \prod_{C,M} K_i \right)/ \left(\prod_C F \right)\right)} \\
	\\
	0 && {Ker (\eta_{N})} && {\prod_{C, N} G_i} && {Gal \left( \text{Alg} \left( \prod_{C,N} K_i \right) / \left( \prod_C F \right) \right)}
	\arrow[from=1-1, to=1-3]
	\arrow[hook, from=1-3, to=1-5]
	\arrow[from=1-3, to=3-3]
	\arrow["{\eta_M}", from=1-5, to=1-7]
	\arrow["{f_{M,N}}", from=1-5, to=3-5]
	\arrow[from=1-7, to=3-7]
	\arrow[from=3-1, to=3-3]
	\arrow[hook, from=3-3, to=3-5]
	\arrow["{\eta_N}", from=3-5, to=3-7]
\end{tikzcd}\]

  For those unfamiliar with projective limits, or uncomfortable with the category theoretic language, we present a direct proof of the result:

\begin{proof}

The induced map $$\Theta : \varprojlim_{N\in \prod_C \Nb} \prod_{C,N} G_i\to \varprojlim_{N\in \prod_C \Nb } Gal(\text{Alg}(\prod_{C,N} K_i)/\prod_CF)$$ sends $\{ g_N\}_{N\in \prod_C \Nb} $ to $\{ \eta_Ng_N\}_{N\in \prod_C \Nb}$, where $$\eta_N: \prod_{C,N} G_i \to Gal(\text{Alg}(\prod_{C,N} K_i )/\prod_C F)$$ is as in Theorem \ref{Bad Description of the Kernel}.  Theorem \ref{Cohesive Product of Actions Commutes with Quotients} guarantees that this map is well-defined. That is, that $\{ \eta_Ng_N\}_{N\in \prod_C \Nb}$ respects the compatibility criteria for membership in $\varprojlim_{N\in \prod_C \Nb } Gal(\text{Alg}(\prod_{C,N} K_i)/\prod_CF)$.

To show that the image is dense, it is sufficient to show that the image has fixed field $\prod_C F$ by Lemma \ref{same fixed field iff same closure}.  Let $\alpha \in \text{Alg} (\prod_C K) - \prod_C F$.  Since $\prod_C K = \bigcup_{N \in \prod_C \Nb } \left (\prod_{C,N} K_i \right)$, we have that $\alpha \in \text{Alg}(\prod_{C,N} K_i)$ for some $N$.  By Theorem \ref{Bad Description of the Kernel}, there is some $g_N \in \prod_{C,N} G_i$ that acts nontrivially on $\alpha$.  Extend $g_N$ to $g \in \varprojlim_{N\in \prod_C \Nb}\prod_{C,N} G_i  $ arbitrarily.  By definition, $g\alpha = \Theta(g)\alpha= \eta_N(g)\alpha=g_N \alpha \neq \alpha$.  Since $\alpha$ was arbitrary, the image is dense.  

To show that $ker \Theta = \varprojlim_{N\in \prod_C \Nb} ker (\eta_N)$, simply notice that $g\in \varprojlim_{N\in \prod_C \Nb}\prod_{C,N} G_i $ acts trivially on $\text{Alg}(\prod_C K)$ if and only if $g_N$ acts trivially on $\text{Alg}(\prod_{C,N} K_i)$ for every $N \in \prod_C \Nb$. 
    
\end{proof}

As with Theorem \ref{Bad Description of the Kernel}, this result will be re-characterized in a more legible way in section \ref{hyp}.

\section{Hyper-Automorphism Groups} \label{hyp}

There is a great deal of hidden extra structure in a cohesive power.  For example, consider the sequence $\{\sqrt[i]{2}\}_{i\in \omega}$.  When the polynomials $\{ x^i -2\}_{i\in \omega}$ are applied to the sequence term by term, the sequence is identically $0$.  The sequence of numbers is naturally an element of $\prod_C \overline{\Qb}$, and the sequence of polynomials is naturally an element of $\prod_C \Zb[x]$.  This suggests an ``evaluation-like map'' $\prod_C \Zb [x] \curvearrowright \prod_C \overline{\Qb} $.  

Further, if we consider the sequences $\{ \sqrt[2^i]{2} \}_{i\in \omega}$ and $\{ \sqrt[3^i]{2}\}_{i\in \omega}$ acted upon by the sequence of polynomials $\{ x^{2^i} - y^{3^i} \}_{i \in \omega}$, the resulting sequence is identically $0$.  As before, this suggests an ``evaluation-like map'' $\prod_C \Zb[x,y] \times \left( \prod_C \overline\Qb  \right)^2 \to \prod_C \overline \Qb$.

 Similarly, there is an action $\Zb [x_1,x_2,\dots] \curvearrowright \overline\Qb $.  This suggests an action from $\prod_C \Zb[x_1,x_2,\dots  ]$ onto $\prod_C \overline{\Qb}$. The action $ \Zb [x_0,x_1,x_2,\dots] \curvearrowright \overline\Qb$ requires us to know how many variables there are in a given polynomial, and to input the correct number of arguments.  However, an element of $\prod_C \Zb[x_0,x_1,\dots]$ can have an infinite number of variables.  An example of this is $p=\{ x_0x_1\cdots x_i\}_{i\in \omega} $, as $\prod_C \Zb[x_0,x_1,x_i,\dots] \models x_i \mid p$ for every $i$. Such an action would need to take as input some infinite collection of elements from $\prod_C \overline{\Qb}$.  However, not just any infinite collection will suffice.  We need a notion of a non-standard tuple, but a non-standard tuple requires some notion of non-standard arity, which requires the development of suitable machinery.

\begin{Definition}
    Let $\mathcal{A}$ be a computable structure.  Let $\mathcal{T}_\mc{A}$ denote the tree of tuples of elements of $A$, together with the extension order and concatenation operation.  That is, $\mathcal{T}_\mathcal{A} = ( \mathcal{A}^{< \omega}, \preceq , ^ \wedge) $.
\end{Definition}

\begin{Proposition}\label{Prop53}

    Let $\{\mathcal{A}_i\}_{i\in \omega}$ be a uniformly computable sequence of structures.  The following maps exist and have the following properties.

    \begin{enumerate}
        \item A non-standard arity map, $\text{arity} : \prod_C\mathcal{T}_{\mathcal{A}_i} \to \prod_C \Nb$.

         \begin{enumerate}

             \item For $\vec{a}, \vec{b} \in \prod_C \mathcal{T}_{\mathcal{A}_i}$, $\text{arity} (\vec{a} ^\wedge \vec{b} ) = \text{arity}(\vec{a}) + \text{arity}(\vec{b})$.

              \item For $\vec{a},\vec{b} \in \prod_C \mc{T}_{\mc{A}_i}$, if $ \vec{a} \preceq \vec{b}$ then $arity(\vec{a}) \leq arity (\vec{b})$.
         \end{enumerate}

        \item A partial map that ``retrieves entries,'' $\text{entry}:\prod_C \mathcal{T}_{\mathcal{A}_i} \times \prod_C \Nb \to \prod_C \mathcal{A}_i$.

        \begin{enumerate}        
            \item $\text{entry} (\vec{a}, I)$ is defined if and only if $I < \text{arity}(\vec{a})$.

            \item If $I < \text{arity}(\vec{a}),$ and $\vec{a} \preceq \vec{b}$, then $\text{entry}(\vec{a},I ) = \text{entry}(\vec{b},I)$.

            \item If $I < \text{arity}(\vec{b})$, then $\text{entry}( \vec{b},I) = \text{entry}(\vec{a}^\wedge \vec{b}, \text{arity}(\vec{a})+I )$.       
        \end{enumerate}

        \item For $\vec{a},\vec{b} \in \prod_C \mathcal{T}_{\mathcal{A}_i}$, $\vec{a}=\vec{b}$ if and only if $artiy(\vec{a})=arity(\vec{b})$ and for every $ I < arity(\vec{a})$,  we have $ entry(\vec{a},I)=entry(\vec{b},I)$.
        
    \end{enumerate}
    
\end{Proposition}

\begin{proof}

\hspace{1mm}

    \begin{enumerate}
        \item First, we define the map.  For $\vec{a}\in \prod_C \mathcal{T}_{\mathcal{A}_i}$, $(\text{arity}(\vec{a}))(i)= |\vec{a}(i)|$.  This map is well-defined, as the arity of a tuple $ b \in \mathcal{T}_{\mathcal{A}_i}$ can be found computably by counting the elements in the list, and if two sequences have  $\vec{a}(i)=\vec{b}(i)$ for almost every $i$, then $|\vec{a}(i)|=|\vec{b}(i)|$ for almost every $i$.

        \begin{enumerate}
            \item Let $\vec{a}, \vec{b} \in \prod_C \mathcal{T}_{\mathcal{A}_i}$.  Then
            
            $\prod_C \Nb \models \text{arity} ( \vec{a}^\wedge \vec{b} ) = \text{arity}(\vec{a})+ \text{arity}(\vec{b})$
            
            $\iff \{ i \mid \text{arity}( \vec{a}^\wedge \vec{b}) (i) = \text{arity}(\vec{a})(i) + \text{arity}(\vec{b})(i) \} \supseteq^* C $
            
            $\iff \{ i \mid |\vec{a}(i)^\wedge \vec{b}(i)|=|\vec{a}(i)|+|\vec{b}(i)| \} \supseteq^* C $
            
            $\iff dom(\vec{a}) \cap dom (\vec{b}) \supseteq^* C$

            $\iff \vec{a} , \vec{b} \in \prod_C T_{\mathcal{A}_i}$

            \item If $\vec{a} \preceq \vec{b}$, then there is some $\vec{c}$ with $\vec{a} ^\wedge \vec{c} = \vec{b}$.  Thus, $arity(\vec{a}) + arity(\vec{c}) = arity (\vec{b})$.  Therefore $arity(\vec{a}) \leq arity (\vec{ b} )$.
            
        \end{enumerate}

        \item Let $(\text{entry}( \vec{a}, I ) )(i)=((\vec{a})(i))_{I(i)} $.  That is, $\text{entry}(\vec{a},I)(i)$ is the $I(i)^{\text{th}}$ entry of $\vec{a}(i) \in \mathcal{T}_{\mathcal{A}_i}$.  Because entry retrieval is computable, the map exists and is well-defined.

        \begin{enumerate}
            \item Now let $I \in \prod_C \Nb$.  We have $I < arity(\vec{a})$.
            
            $\iff \{ i \mid I(i) < |\vec{a}(i) | \}\supseteq^* C $
           
            $\iff \{ i \mid (\vec{a}(i))_{I(i)} \text{ exists} \} \supseteq^* C$
           
            $ \iff  \{ i \mid entry(\vec{a},I) (i) \downarrow \} \supseteq^* C$
           
            $\iff entry(\vec{a},I) \in \prod_C \mathcal{A}_i$ exists.

            \item Let $I < arity(\vec{a})$, and let $\vec{a} \preceq \vec{b}$. 

            $ entry(\vec{a}, I ) = entry(\vec{b},I)$

            $\iff \{ i \mid entry(\vec{a},I) (i) = entry (\vec{b},I)(i) \} = \{ i \mid \vec{a}(i)_{I(i)} = \vec{b}(i)_{I(i)}\}\supseteq^* C$

           Clearly $\{ i \mid \vec{a}(i)_{I(i)} = \vec{b}(i)_{I(i)}\} \supseteq \{ i \mid I(i) < |\vec{a}(i)|\} \cap \{ i \mid \vec{a}(i) \preceq \vec{b(i)} \}$.

            The latter set is large on $C$ by the fundamental theorem of cohesive products, as $\prod_C \Nb \models I < arity( \vec{a})$, and $\prod_C\mathcal{T}_{\mathcal{A}_i} \models\vec{a} \preceq \vec{b}$.

            \item Similar to (b).
        \end{enumerate}

        \item ($\Rightarrow$)  Trivial.

        ($\Leftarrow$) Let $\vec{a}, \vec{b} \in \prod_C \mathcal{T}_{\mathcal{A}_i}$ be distinct.  If $arity(\vec{a})\not=arity(\vec{b})$, then $\vec{a}\not=\vec{b}$, as $arity$ is a function. Therefore, assume further that $\text{arity}(\vec{a}) = \text{arity}( \vec{b})$.  Since $\vec{a} \not= \vec{b}$, we have $\{ i \mid \vec{a}(i) \not=\vec{b(i)}\}\supseteq^* C$.  Let $I(i) = $ be the least $x$ such that $ (\vec{a}(i))_x \not = (\vec{b}(i))_x$.  $I(i) \downarrow$ if and only if $\vec{a}(i) \not= \vec{b(i)}$.  Thus, $\text{dom}(I) \supseteq^* C$.  Therefore, $I \in \prod_C \Nb$.  By construction, $entry(\vec{a},I) \not= entry(\vec{b},I)$.
        
    \end{enumerate}
\end{proof}

The proposition illustrates that the elements of $\prod_C \mathcal{T}_{\mathcal{A}_i}$ are simply lists of elements of $\prod_C \mathcal{A}_i$, indexed over initial segments of $\prod_C \Nb$, and the internal evaluation and arity maps are as they should be.  Thus, we will use the more legible notation $|\vec{a}|$ and $(\vec{a})_I$ to denote $arity(\vec{a})$ and $entry(\vec{a},I)$.  The context will make clear in which domain the tuples reside.  We now have a collection of ``hyper-finite'' tuples, internal to the logic of a cohesive power.  This internal class of \emph{hyper-tuples} allows to take hyper-finite sums and products, among other things. 

\begin{Proposition}
    Let $\{K_i\}_{i\in \omega}$ be a uniformly computable sequence of fields.  The following maps exist and have the following properties. 

    \begin{enumerate}
        \item A hyper-finite ``sum'' function $\Sigma : \prod_C\mathcal{T}_{K_i} \to \prod_C K_i$, where

        \begin{enumerate}
            \item For $\vec{a} \in \prod_C \mathcal{T}_{K_i}$, with $| \vec{a}| \in \Nb$ a standard number, $\Sigma \vec{a} = \sum_{j < |\vec{a}|} (\vec{a})_j$;

            \item For $\vec{a} , \vec{b} \in \prod_C \mathcal{T}_{K_i}$, $\Sigma (\vec{a} ^\wedge \vec{b}) = \Sigma \vec{a} + \Sigma \vec{b}$.
        \end{enumerate}

        \item A hyper-finite ``product'' function $\Pi : \prod_C \mathcal{T}_{K_i} \to \prod_C K_i$, where

         \begin{enumerate}
            \item For $\vec{a} \in \prod_C \mathcal{T}_{K_i}$, with $|\vec{a}| \in \Nb$ a standard number, $\Pi \vec{a} = \prod_{i < |\vec{a}|} (\vec{a})_i$;

            \item For $\vec{a} , \vec{b} \in \prod_C \mathcal{T}_{K_i}$, $\Pi (\vec{a} ^\wedge \vec{b}) = (\Pi \vec{a} )( \Pi \vec{b})$;

            \item For $\vec{a} \in \prod_C \mathcal{T}_{K_i}$, $\Pi \vec{a} = 0 $ if and only if there is some $I \in \prod_C \Nb$ with $\prod_C K_i \models (\vec{a} )_I = 0$.
        \end{enumerate}

    \end{enumerate}
    
\end{Proposition}

\begin{proof}

\hspace{1mm}

    \begin{enumerate}
        \item $(\Sigma \vec{a} )(i) = \sum_{j < | \vec{a}(i)| } (\vec{a})_j$.  The map is well defined as addition in the $K_i$'s is uniformly computable. 

        \begin{enumerate}
            \item Let $\vec{a} \in \prod_C \mathcal{T}_{K_i}$, with $|\vec{a}| = n \in \Nb$.  Thus, $\{ i \mid |\vec{a}(i)| = n \} \supseteq^*C$.

            By the fundamental theorem of cohesive products, $ \prod_C K_i \models \Sigma \vec{a} = (\vec{a})_0+\dots  + (\vec{a})_{n-1}$ if and only if $\{ i \mid (\Sigma\vec{a})(i)= (\vec{a})_0(i)+\dots  +(\vec{a})_{n-1}(i)\} \supseteq^* C$.

            The latter set is equal to $\{ i \mid \sum_{j < |\vec{a} (i)| } (\vec{a}(i))_j = \vec{a}(i)_0+\dots  +\vec{a}(i)_{n-1} \}$.  This set is possibly not equal to $\omega$, as it might not be the case that $\vec{a}(i)_{n-1}$ exists for every $i$.  Rather, it is equal to the set $\{ i \mid |\vec{a}(i)| = n \} $, which is large on $C$. 

            \item Let $\vec{a},\vec{b},\vec{c} \in \prod_C \mathcal{T}_{K_i}$ such that $\prod_C \mathcal{T}_{K_i} \models \vec{a} ^\wedge \vec{b} = \vec{c}$.  Then, $\{ i \mid \vec{a}(i) ^\wedge \vec{b}(i) = \vec{c}(i) \} \supseteq^* C$.
            
            $\prod_C K_i \models \Sigma (\vec{c})= \Sigma\vec{a}+ \Sigma\vec{b}$ 

            $\iff \{ i \mid (\Sigma\vec{c})(i)= (\Sigma \vec{a})(i) + (\Sigma\vec{b})(i) \} = \supseteq^*C$

            $\iff \{i \mid \sum_{j < | \vec{c}(i)|}  (\vec{c}(i))_j = \sum_{j < |\vec{a}(i)|} (\vec{a}(i))_j + \sum_{j < |\vec{b}|}(\vec{b}(i))_j  \}\supseteq^* C $

            The set $\{i \mid \sum_{j < | \vec{c}|} (\vec{c}(i))_j = \sum_{j < |\vec{a(i)}|} (\vec{a}(i))_j + \sum_{j < |\vec{b}|}(\vec{b}(i))_j  \}$ contains $\{ i \mid \vec{c}(i) = \vec{a}(i) ^\wedge \vec{b}(i)\}$, which is large on $C$.
        \end{enumerate}

        \item This is similar to (1), replacing sum with product.

        \begin{enumerate}
            \item Similar to (1.a).

            \item Similar to (1.b).

            \item Let $\vec{a} \in \prod_C \mathcal{T}_{K_i}$ such that $\prod_C K_i \models \Pi \vec{a} = 0$. 

            Then $\{ i \mid \prod_{j < |\vec{a}(i)|}  (\vec{a}(i))_j = 0 \}\supseteq^* C$ by the fundamental theorem of cohesive products.  Let $I(i)$ be the least $x$ such that $  (\vec{a}(i))_x = 0$.  Since $K_i$ is a field (hence an integral domain), $I(i) \downarrow $ if and only if $\prod_{ j < | \vec{a}(i) |} (\vec{a}(i))_j = 0$. Thus, $\text{dom}(I) \supseteq^*C$.  By construction, $(\vec{a})_I=0$.
        \end{enumerate}
    \end{enumerate}
\end{proof}

The key algebraic object we need to get access to the extra structure within the cohesive power is the ring of hyper-polynomials.  In particular, the ring of hyper-polynomials with integer coefficients will be of central importance.  

\begin{Definition}
    Let $\{R_i\}_{i\in \omega}$ be a uniformly computable sequence of rings.  The ring of hyper-polynomials over $\prod_C R_i$ is the ring $\prod_C R_i[x_1,x_2,\dots  ]$, where $R_i[x_1,x_2,\dots  ]$ is given the ``canonical'' presentation.  That is, the presentation in which both the degree map and the coefficient map $R_i[x_1,x_2,\dots  ]\times \omega^{<\omega}\to R_i$ are computable, where $(p, (k_1,\dots  ,k_n)) $ is mapped to the coefficient for $x_1^{k_1}x_2^{k_2}\cdots  x_n^{k_n}$ in $p$.
\end{Definition}

Note that  $p \in \prod_C F_i[x]$ will be thought of as both a function of $i$, i.e., $p(i) \in F_i[x]$, and as a ``polynomial function'' of $x$.  That is, $p(a) \in \prod_C F_i$, for $a \in \prod_C F_i$.  To avoid confusion, $i$ will be exclusively used to refer to the situation in which $p$ is a function from $\omega$ to $F_i$'s, and $x,y,\dots$ will be exclusively used as the formal variables in polynomials and hyper-polynomials. 

\begin{Proposition} \label{properties of hyperpolynomails}
   Let $\{ K_i/F_i \}_{i\in \omega}$ be a uniformly computable sequence of field extensions. Let $\{ R_i \}_{i\in \omega}$ be $\{ \Zb \}_{i\in \omega}$, $\{ F_i \}_{i\in \omega}$, or $\{ K_i \}_{i\in \omega}$.  The following maps exist and have the following properties. 

    \begin{enumerate}
        \item A non-standard degree map $\deg : \prod_C R_i [x_1,x_2,\dots  ] \to \prod_C \Nb$

         \begin{enumerate}
             \item For $p,q \in \prod_C R_i  [x_1,x_2,\dots  ]$, $\deg ( p  q ) = \deg(p) + \deg(q)$.

             \item For $p,q \in \prod_C R_i [x_1,x_2,\dots  ]$, $\deg ( p +q ) \leq \text{max}\{\deg (p), \deg(q)\}$.
         \end{enumerate}
    
    \item A variable counting map $\text{var} : \prod_C R_i [x_1,x_2,\dots  ] \to \prod_C \Nb$ which ``returns the index of the largest variable occurring within a given hyper-polynomial.''

    \begin{enumerate}
    
        \item For $p,q \in \prod_C R_i [x_1,x_2,\dots  ]$, $\text{var} (pq) = \max \{ \text{var}(p), \text{var}(q)\}$

        \item For $p,q \in \prod_C R_i [x_1,x_2,\dots  ]$, $\text{var} (p+q) \leq \max \{ \text{var}(p), \text{var}(q)\}$.
    \end{enumerate}
    
    \end{enumerate}
    
\end{Proposition}

We follow the convention that the degree of the $0$ polynomial is $\infty$.

\begin{proof}

\hspace{1mm}

    \begin{enumerate}
        \item Let $p\in \prod_C R_i [x_1,x_2,\dots  ]$.  Set $\deg(p)(i) =\deg(p(i))$. The map exists and is well-defined in the cohesive power, as recognizing the degree of a polynomial in the canonical presentation of $R_i[x_1,x_2,\dots  ]$ is computable uniformly from $R_i$. 

        \begin{enumerate}
            \item Let $p,q,r \in \prod_C R_i[x_1,x_2,\dots  ]$ such that $\prod_C R_i[x_1,x_2,\dots  ] \models pq=r$.  Then $\{i \mid p(i) q(i) = r(i) \} \supseteq^* C $.  The set $\{ i \mid \deg(p(i))+\deg(q(i))= \deg(r(i)) \} \supseteq \{ i \mid p(i)q(i)=r(i)\}$.  Thus, $\prod_C \Nb \models \deg(p) +\deg(q)=\deg(r)$.

            \item Similar to (1.a)
        \end{enumerate}

    \item Let $var(p)(i)$ be the largest index of a variable appearing in $p(i)$. 

    \begin{enumerate}
        \item Similar to (1.a).

        \item Similar to (1.b.)
    \end{enumerate}
        
    \end{enumerate}
\end{proof}

It is helpful to think of hyper-polynomials as hyper-finite sums, indexed over hyper-naturals, of terms of the form $a_i \prod_{I\leq J} x_I^{K_i}$, where $a_i \in \prod_C F_i$, and $\prod_{I\leq J} x_I^{K_i}$ is a hyper-finite product of variables, with indices, coefficients, and endpoints from $\prod_C \Nb$.  

\begin{Proposition}
    The univariate subring $\prod_C K_i [x]$  of $\prod_C K_i [x,x_1,x_2,\dots  ] $ for fields $K_i$, is a hyper-Euclidean domain in the following sense:

 Given $p,q \in \prod_C K_i [x]$, there are $r,s \in \prod_C K_i[x]$ such that $ p = sq+r$, and $\deg (r) < \deg (q)$, that are unique up to scaling by some unit $u \in \prod_C K_i$. 

\end{Proposition}

\begin{proof}
        The argument essentially boils down to ``polynomial division is computable''.

    Let $p,q \in \prod_C K_i[x]$.  For each $i$, let $r(i), s(i)$ be the unique polynomials such that $p(i)= s(i)q(i)+r(i)$.  Note that $r(i), s(i)$ exist whenever both $p(i)$ and $q(i)$ exist.  Thus, $dom(r)=dom(s) = dom(p) \cap dom(q) \supseteq^* C$. It follows from construction that $\prod_C K_i[x] \models p=sq+r$.

    Now $\prod_C \Nb \models \deg( r) < \deg(q)$, also by construction.

    As for uniqueness, let $s',r'$ be such that $\prod_C K_i[x] \models p=s'q+r'$.  Let $u(i)$ be the unique element of $K_i$, such that $s(i)=u(i)s'(i)$, and similarly $v(i) $ such that $r(i)=v(i)r'(i)$.  $u(i)$ exists and is unique by elementary algebra, whenever $s(i)$ and $s'(i)$ exist, and similarly for $v(i)$.  Thus $dom(u) = dom(s)\cap dom(s') \supseteq^*C$.  Similarly $dom(v) = dom(r)\cap dom(r')$.  By construction, $u,v \in \prod_C K_i$ are units, as they are non-zero field elements, and $ \prod_C K_i[x] \models s=us'\wedge r=vr'$.

\end{proof}

Finally, we may define the ``hyper-polynomial structure'' hidden inside a cohesive product of fields.

\begin{Proposition}
    Let $\{ K_i/F_i \}_{i\in \omega}$ be a uniformly computable sequence of field extensions.  Let $\{ R_i \}_{i\in \omega}$ be $\{ \Zb \}_{i\in \omega}$, $\{ F_i \}_{i\in \omega}$, or $\{ K_i \}_{i\in \omega}$.  Then there is a partial map 

    $$ eval: \prod_C R_i[x_1,x_2,\dots  ] \times \prod_C \mathcal{T}_{K_i} \to \prod_C K_i$$,

    which has the following properties:  

    \begin{enumerate}
        \item For $p \in \prod_C R_i[x_1,x_2,\dots  ]$ and $\vec{a} \in \prod_C \mathcal{T}_{K_i}$, $eval(p,\vec{a})$ exists if and only if $\text{var}( p) =  |\vec{a}|$. 

        \item For $p \in \prod_C R_i[x_1,x_2,\dots  ]$ and $\vec{a} \in \prod_C \mathcal{T}_{K_i}$, $eval(p+q, \vec{a} ) = eval(p,\vec{a}) +eval(q,\vec{a})$, and $eval(pq, \vec{a} ) = \left(eval(p,\vec{a})\right) \left(eval(q,\vec{a})\right)$. 
    \end{enumerate}
\end{Proposition}

\begin{proof}

Given $p \in \prod_C R_i[x_1,x_2,\dots  ]$, and $\vec{a} \in \prod_C \mathcal{T}_{K_i}$, set $$eval(p,\vec{a})(i) = \begin{cases}
            p(i)(\vec{a}(i)) & \text{ if } |\vec{a}(i)| = var(p(i)) \\
            \uparrow & \text{ otherwise}
        \end{cases}$$

    \begin{enumerate}

       \item If $|\vec{a}| = var(p)$, then $\{ i \mid |\vec{a}(i)| =var(p(i))\} \supseteq^* C$.  Since $\{ i \mid |\vec{a}(i) | = var(p(i)) \} = \{ i \mid eval(p,\vec{a})(i)\downarrow\}$, $eval(p,\vec{a})$ is a well-defined element of $\prod_C K_i$ if and only if $var (p) = |\vec{a}|$.

       \item The proof is straightforward.
    \end{enumerate}
\end{proof}
Where appropriate, the prefix hyper-, and words like non-standard or pseudo-finite, will be used to make it clear in which structure objects reside. For example, elements of $K[x]$ are polynomials, while elements of $\prod_C (K[x])$ are hyper-polynomials.  Similarly, elements of $\mathcal{T}_A$ are finite tuples, while elements of $\prod_C \mathcal{T}_A$ are pseudo-finite tuples.  

Where expedient, we will abuse notation, and write things such as $p(\vec{a}) =b$ to mean $eval(\vec{p},\vec{a}) = b$ via $eval:\prod_C K[x_1,x_2,\dots  ] \times \prod_C \mathcal{T}_K\to \prod_C K$.

\begin{Proposition}
    Let $\{K_i/F_i\}_{i\in \omega}$ be a uniformly computable sequence of algebraic field extensions.  Then for every $\alpha \in \prod_C K_i $, there is some $p \in \prod_C ( F[x])$ with $p(\alpha) = 0$.
    
    Moreover, if the minimal polynomial maps $K_i \to F[x]$ are uniformly computable, then every $\alpha\in \prod_C K_i$ has a ``minimal hyper-polynomial''.  That is, there is a unique $p_\alpha \in \prod_C (F[x])$ such that for every other $q \in \prod_C (F[x])$, if $q(\alpha) = 0$, then there is a unique $r \in \prod_C F[x]$ with $q = p_\alpha \cdot r$.
\end{Proposition}

\begin{proof}

Let $\alpha\in \prod_C K_i$.  For each $i\in \omega$, let $p(i)$ be the first encountered element of $K_i[x]$ with $p(i) ( \alpha(i)) = 0$. The function $p$ is clearly computable and defined whenever $\alpha$ is.  Thus $p$ is an element of $\prod_C F[x]$.  By construction, $p(\alpha) = 0$.

Now suppose further that the minimal polynomial maps are uniformly computable.  Let $p_\alpha(i) = p_{\alpha(i)}$. That is, $p_\alpha(i)$ is the minimal polynomial of $\alpha(i)$. Now, suppose there is $ q \in \prod_C F[x]$ with $q(\alpha) = 0$.  By the fundamental theorem of cohesive products, we have $q(i) (\alpha(i)) = 0$ for almost every $i$.  For such $i$, we may effectively find the unique $r(i)$ with $q(i) = p_\alpha(i) r(i)$.  Thus, $\prod_C F[x] \models q= (p_\alpha) (r)$.
    
\end{proof}

In nice families of extensions $K_i/F$ where elements of $\prod_C K_i$ have minimal hyper-polynomials over $\prod_C F$, we also have a notion of non-standard \textit{degree}.  The degree map is found by composing the minimal hyper-polynomial map with the non-standard degree map from Proposition \ref{properties of hyperpolynomails}  

With a notion of hyper-degree, we may build a notion of ``hyper-Galois.''

\begin{Proposition}
    Let $\{K_i / F\}_{i\in\omega}$ be uniformly computable field extensions, with each $K_i/F$ Galois, and uniformly computable minimal polynomial maps.  Then, $\prod_C K_i /\prod_C F$ is hyper-Galois in the following sense. 

    For each $ a \in \prod_C K_i$ with minimal hyper-polynomial $p_a \in \prod_C(F[x])$, there is a $ \vec{b} \in \prod_C \mathcal{T}_{K_i}$ with the following properties: 
    
    \begin{enumerate}
        \item $|\vec{b}| = \deg (p_a)$
        \item $(\vec{b})_0= a$
        \item for $I <J < |\vec{b}|$, we have $(\vec{b})_I \not= (\vec{b})_J$
        \item for every $c\in \prod_C K_i$, we have $p_a(c) =0$ if and only if $c= (\vec{b})_I$ for some $I$.
    \end{enumerate}
\end{Proposition}

\begin{proof}

Let $a \in \prod_C K_i$, and let $ p_a $ be the corresponding minimal hyper-polynomial.  For each $i$, let $\vec{b}(i) $ be the first encountered element of $\mathcal{T}_{K_i}$ with $ \vec{b}(i)_0 = a(i)$, and $p_a(i) = \prod_{j < |b(i)|} (y-b(i)_j)$.  Such a $b(i)$ exists as $K_i$ is Galois. 

The rest follows immediately from the fundamental theorem of cohesive products.
    
\end{proof}

\begin{Definition}
    Let $\{\mathcal{A}_i \}_{i\in\omega}$ be a uniformly computable sequence of structures.  The hyper-automorphism group of $\prod_C \mathcal{T}_{\mathcal{A}_i}$, denoted $HAut( \prod_C \mathcal{T}_{\mathcal{A}_i})$, is the subgroup of $Aut (\prod_C \mathcal{T}_{\mathcal{A}_i}, \preceq,^\wedge)$ that preserves the arity map $|\cdot| : \prod_C \mathcal{T}_{\mathcal{A}_i} \to \prod_C \Nb$.  That is, 

    $$ HAut \left(\prod_C \mathcal{T}_{\mathcal{A}_i}\right) = \left\{ \sigma \in Aut\left(\prod_C \mathcal{T}_{\mathcal{A}_i}\right) \middle| (\forall \vec{a} \in \prod_C \mathcal{T}_{\mathcal{A}_i})(|\vec{a}|=|\sigma \vec{a}| )\right\}$$
\end{Definition}

\begin{Proposition}\label{presentry}
    $HAut(\prod_C \mathcal{T}_{\mathcal{A}_i})$ preserves the entry maps.  That is, for every $\vec{a} \in \prod_C \mathcal{T}_{\mathcal{A}_i}$, $I < |\vec{a}|$, and $\sigma \in HAut(\prod_C \mathcal{T}_{\mathcal{A}_i} )$, $\sigma((\vec{a})_I) = (\sigma(\vec{a}))_I$, where $(\vec{d} )_I$ is thought of as a 1-tuple in $\prod_C \mathcal{T}_{\mathcal{A}_i}$.
\end{Proposition}

\begin{proof}
    Let $\vec{a}$, $I$, and $\sigma$ be as above. Then there are unique $b,c \in \prod_C \mathcal{T}_{\mathcal{A}_i}$ such that $\prod_C \mathcal{T}_{\mathcal{A}_i} \models \vec{a} = \vec{b}^\wedge ((\vec{a})_I )^\wedge \vec{c}$ by the fundamental theorem of cohesive products. (Existence is $\Sigma_1^0$ and uniqueness is $\Pi_1^0$.)  $(\sigma( \vec{a}))_I = (\sigma (\vec{b})^ \wedge \sigma((\vec{a})_I)^\wedge \sigma(\vec{c}))_I= \sigma((\vec{a})_I)$.
\end{proof}

\begin{Corollary}\label{corhaut}
    Two elements $\sigma ,\tau \in HAut(\prod_C \mathcal{T}_{\mathcal{A}_i})$ are equal if and only if they agree on all 1-tuples.
\end{Corollary}

\begin{proof}
    This follows from Propositions \ref{presentry} and \ref{Prop53} (3).
\end{proof}

\begin{Definition}
    Let $K_i/F$ be a uniformly computable sequence of field extensions.  The group of hyper-automorphisms of $\prod_C K_i /\prod_C F$, denoted $HAut ( \prod_C K_i  )$, is the group 

    $$\left \{ \sigma \in HAut \left( \prod_C \mathcal{T}_{K_i} \right) \middle| \forall p \in \prod_C (\Zb[x_1,x_2,\dots  ] ) \forall \vec{a} \in \prod_C \mathcal{T}_{K_i} , [ |\vec{a}| = deg(p) \implies p(\sigma (\vec{a}) )= \sigma (p(\vec{a}))]\right \}$$

Further, 

$$HAut \left( \prod_C K_i  /\prod_C F \right) = \left \{  \sigma \in HAut \left( \prod_C K_i \right) \middle| (\forall \vec{a} \in \prod_C \mathcal{T}_F) [\sigma(\vec{a}) =\vec{a}] \right \}$$

\end{Definition}

While it is strange that the hyper-automorphism group of a structure is a subgroup of the automorphism group of its tree of hyper-tuples, we must preserve the non-standard arities of tuples.  Classical automorphisms do this for free, since bijections preserve cardinality.  However, we must impose preservation of non-standard size from without.  A structure is canonically a subset of its tree of tuples by restricting to the set of $1$-tuples, so the automorphisms of the tree restrict to automorphisms of the field.  

The following lemma shows that the hyper-automorphism group $HAut(\prod_C K_i /\prod_C F )$ preserves univariate hyper-polynomial functions from $\prod_C F[x]$.

\begin{Lemma}\label{commutes-univ-hyper}
Let $\{K_i/F\}_{i\in \omega}$ be a uniformly computable sequence of field extensions.  Let $\sigma \in HAut(\prod_C K_i /\prod_C F)$,  let $p \in \prod_C F[x_0]$, and let $ a\in \prod_C K_i$.  Then $\sigma(p(a))= p(\sigma(a))$.
\end{Lemma}

\begin{proof}
     For each $i$, let $p(i)(x_0) = q(i)(x_0,\vec{c}(i))$, where $q(i)(x_0,y_0,\dots  ,y_{N(i)-1}) = \sum_{ j < N(i)} (y_{j}) (x_0)^{j} = y_{N(i)-1} x_0^{N(i)-1}+y_{N(i)-2}x_0^{N(i)-2}+\cdots+ y_0 $, and $\vec{c}(i) \in \mathcal{T}_{K_i}$ is the vector of coefficients of $p(i)$.

    Note that $\vec{c} \in \prod_C \mathcal{T}_{F}$, and $q(x_0,\vec{y}) \in \prod_C \Zb[x_0,y_0,y_1,\dots  ]$.

    Now $\prod_C K_i \models q(a, \vec{c}) = b$ by the fundamental theorem of cohesive products.  Therefore, $\sigma(b) = \sigma(p(a))=\sigma(q(a,\vec{c})) = q(\sigma(a),\sigma(\vec{c}))= q(\sigma(a),\vec{c}) =p(\sigma(a))$. Note that $\sigma (\vec{c}) = \vec{c}$, as $\sigma \in HAut(\prod_C K_i / \prod_C F)$, and as $\vec{c} \in \prod_C T_F$.
\end{proof}

\begin{Lemma}\label{pflemma}

Let $\{K_i/F\}_{i\in \omega}$ be a uniformly computable sequence of finite Galois extensions with uniformly computable Galois groups and actions $\{ G_i \curvearrowright K_i\}_{i\in \omega}$.  Let $ \alpha (i)$ be a primitive generator of $K_i /F$.  

An element $\sigma \in HAut ( \prod_C K_i / \prod_C F)$ is determined uniquely by the value of $\sigma (\alpha)$.

\end{Lemma}

\begin{proof}
    We have that $\alpha(i)$ is computable via \ref{tfaegal}.  Let $ b\in \prod_C K_i$ be arbitrary. Let $N(i) = [K_i:F]$.  Then $b(i)$ may be expressed uniquely as $ p(i)(\alpha(i))$, where $p(i) \in F[x_0]$ of degree equal to $N(i)-1$, whenever $b(i)$ exists.  This $p(i)$ may be found computably by searching through $F[x_0]$ and evaluating every polynomial on $\alpha(i)$.  

    By \ref{commutes-univ-hyper} $\sigma (b) = \sigma (p(\alpha)) = p( \sigma(\alpha))$.  By \ref{corhaut}, this uniquely determines $\sigma$.   
\end{proof} 

It is interesting to note that $N\in \prod_C \Nb$ is the non-standard degree of the pseudofinite extension $\prod_C K_i / \prod_C F$.

\begin{Theorem} \label{Thm-HAut-is-Product}
    Let $\{K_i /F\}_{i\in \omega} $ be a uniformly computable sequence of finite Galois extensions with Galois groups $G_i$ uniformly computable, and Galois actions uniformly computable.  Then $HAut ( \prod_C K_i / \prod_C F) \cong \prod_C G_i$.  
\end{Theorem}

\begin{proof}

First note the Galois actions can be extended to an action $G_i \curvearrowright \mathcal{T}_{K_i}$ by acting coordinate-wise.  The map $\prod_C G_i \to HAut(\prod_C \mathcal{T}_{K_i})$ is the cohesive product of the extended Galois actions $G_i \curvearrowright \mathcal{T}_{K_i}$.  That is, $g \in \prod_C G_i$ is mapped to $\psi_g$, where 
$$\left(\psi_g \left(\vec{a}\right)\right)\left(i\right)= g(i)(\vec{a}(i))= \left(g\left(i\right)\left(\vec{a}\left(i\right)\right)_0,g\left(i\right)\left(\vec{a}\left(i\right)\right)_1,\dots  ,g\left(i\right)\left(\vec{a}\left(i\right)\right)_{|\vec{a}(i)|-1}\right)$$

$\psi_g \curvearrowright \prod_C \mathcal{T}_{K_i}$ preserves $\preceq, ^\wedge$ and arity, as $g(i) \curvearrowright \mathcal{T}_{K_i}$ preserves $\preceq, ^\wedge$ and arity for every $i$. Thus we conclude $\psi_g $ is a well defined element of $HAut(\prod_C \mathcal{T}_{K_i})$.

Further, 

$$\psi_g (\psi_h(\vec{a}))(i)=\psi_g(h(i)(\vec{a}(i))=g(i)(h(i)(\vec{a}(i)))=((g\circ h)(i))(\vec{a}(i))=\psi_{g\circ h} (\vec{a})(i)$$

Thus, the function $\prod_C G_i \to HAut(\prod_C \mathcal{T}_{K_i})$ is a group homomorphism.  

As we now have a well defined group action $\prod_C G_i \curvearrowright \prod_C \mathcal{T}_{K_i}$, we will write $h(\vec{a}) = \vec{b}$ to mean $\psi_h(\vec{a})=\vec{b}$.  Note that this action agrees on 1-tuples with the action from \ref{Bad Description of the Kernel}. 

Now, we will show that $\psi_g \in HAut(\prod_C K_i)$.  That is, for integral hyper-polynomials $p\in \prod_C \Zb[x_0,x_1,\dots  ]$ and $\vec{a} \in \prod_C \mathcal{T}_{K_i}$ with $var(p)=|\vec{a}|$, $\psi_g(p(\vec{a})) = p(\psi_g(\vec{a}))$.  This is true if and only if $\left\{ i | g\left(i\right)\left(p\left(i\right)\left(\vec{a}\left(i\right)\right)\right)= p\left(i\right)\left(g\left(i\right)\left(\vec{a}\left(i\right)\right)\right)\right\} \supseteq^* C$.

We have that  $g(i)$ always commutes with integral polynomials, as its action is a field automorphism.  Thus
$$\left\{ i | g\left(i\right)\left(p\left(i\right)\left(\vec{a}\left(i\right)\right)\right)= p\left(i\right)\left(g\left(i\right)\left(\vec{a}\left(i\right)\right)\right)\right\} = dom(g)\cap dom(p)\cap dom(\vec{a})\cap \{ i \mid var(p(i))=\vec{a}(i) \} \supseteq^*C$$

Each of the latter sets is large, as $g\in \prod_C G_i$, $p\in \prod_C \Zb[x_0,x_1,\dots  ], \vec{a} \in \prod_C \mathcal{T}_{K_i}$ and $\Nb \models var(p) = |\vec{a}|$. 

Further, $\psi_g \in HAut( \prod_C K_i / \prod_C F)$.  To see this, note that for $\vec{b} \in \prod_C \mathcal{T}_F$, 

$$\psi_g ( \vec{b})(i) = g(i) \vec{b}(i) = \vec{b}(i)$$

whenever $\vec{b}(i)$ and $g(i)$ are both defined, which happens for almost every $i$. 

We show that the kernel of this map is trivial.  Suppose $\psi_g$ is trivial.  Then $\psi_g(\alpha) = \alpha$, where $\alpha(i)$ is the primitive generator for $K_i/F$.  Thus, $g(i) \alpha(i) = \alpha(i)$ for almost every $i$.  As $\alpha(i)$ is a generator for $K_i$, we must have $g(i) = e_{G_i}$ for almost every $i$.  That is, $\prod_C G_i \models g = e$.  

Now, we establish surjectivity.  Let $\sigma \in HAut ( \prod_C K_i /\prod_C F)$, and let $\beta=\sigma(\alpha)$.  Let $p_\alpha \in \prod_C F[x]$ be the minimal hyper-polynomial of $\alpha$.  That is, $p_\alpha(i)$ is the minimal polynomial of $\alpha(i)$ for almost every $i$.  Since, by Lemma \ref{commutes-univ-hyper}, $0=\sigma(0)=\sigma(p_\alpha(\alpha)) = p_\alpha(\sigma(\alpha))=p_\alpha(\beta)$ , we must have also $p_\alpha(i) (\beta(i))= 0$ for almost every $i$.  Thus, $\beta(i)$ is a Galois conjugate of $\alpha(i)$ for almost every $i$.  

Knowing this, we define $h(i)$ to be the unique element of $G_i$ with $h(i)\alpha(i) = \beta(i)$. We may compute $h(i)$ by searching through $G_i$ for some $k$ with $k\alpha(i)=\beta(i)$.  This is an effective search, since $G_i \curvearrowright K_i$ is uniformly computable. By construction, $h (\alpha) = \beta=\sigma(\alpha)$.  Since $\psi_h(\alpha) =h(\alpha)$, Lemma \ref{pflemma} gives us that $\psi_h = \sigma$. The map is thus an isomorphism.  
    
\end{proof}

The naturalness of the $HAut$ groups is clear from the above result, in contrast to Theorem \ref{Bad Description of the Kernel}.  

Now we consider infinite extensions. 

\begin{Lemma} \label{HAut-commutes-with-restriction}
    Let $F \subseteq K_0\subseteq\cdots\subseteq K$ be a fully computable, infinite Galois extension in the sense of Definition \ref{def-fully-computably}.  Let $N,M \in \prod_C \Nb$ with $N < M$.  Let $ \prod_{C,M} G_i \to \prod_{C,N} G_i$ be the induced quotient map as in Theorem \ref{Thm-Gal-Extension-Commutes}.  The following diagram commutes:

    % https://q.uiver.app/#q=WzAsNCxbMCwwLCJcXHByb2Rfe0MsTX1HX2kiXSxbMCwyLCJcXHByb2Rfe0MsTn0gR19pIl0sWzIsMiwiSEF1dChcXHByb2Rfe0MsTn0gS19pLyBcXHByb2RfQyBGKSAiXSxbMiwwLCJIQXV0KFxccHJvZF97QyxNfSBLX2kvIFxccHJvZF9DIEYpICJdLFswLDFdLFsxLDIsIlxcc2ltZXEiLDAseyJsZXZlbCI6Miwic3R5bGUiOnsiaGVhZCI6eyJuYW1lIjoibm9uZSJ9fX1dLFswLDMsIlxcc2ltZXEiLDIseyJsZXZlbCI6Miwic3R5bGUiOnsiaGVhZCI6eyJuYW1lIjoibm9uZSJ9fX1dLFszLDJdXQ==
\[\begin{tikzcd}
	{\prod_{C,M}G_i} && {HAut(\prod_{C,M} K_i/ \prod_C F) } \\
	\\
	{\prod_{C,N} G_i} && {HAut(\prod_{C,N} K_i/ \prod_C F) }
	\arrow["\simeq"', equals, from=1-1, to=1-3]
	\arrow[from=1-1, to=3-1]
	\arrow[from=1-3, to=3-3]
	\arrow["\simeq", equals, from=3-1, to=3-3]
\end{tikzcd}\]

That is to say, the induced quotient map $\prod_{C,M} G_i \to \prod_{C,N} G_i$ is the same as the restriction map $HAut (\prod_{C,M} K_i/ \prod_C F) \to HAut(\prod_{C,N} K_i/\prod_C F)$, via the isomorphism from Theorem \ref{Thm-HAut-is-Product}.

\end{Lemma}

\begin{proof}
    For finite numbers $n,m \in \Nb$ with $n<m$, we let $f_{m,n} : G_m \to G_n$ be the canonical quotient map.  For infinite numbers $N,M \in \prod_C \Nb$, with $N<M$, $f_{M,N} : \prod_{C,M} G_i \to \prod_{C,N} G_i$ is as in Theorem \ref{Thm-Gal-Extension-Commutes}.  That is to say, $f_{M,N}(g)(i)= f_{M(i),N(i)}(g(i))$.

    Let $\vec{a} \in \prod_{C,N} \mathcal{T}_{K_i}$ which is a subset of $\prod_{C,M}\mathcal{T}_{K_i}$ by Lemma \ref{sequence-extension-lemma}, and let $g \in \prod_{C,M}G_i$.  It is sufficient to show that $g \vec{a} = (f_{M,N}(g))\vec{a}$.  This is true, as the set $\{ i \mid g(i) \vec{a}(i)=f_{M(i),N(i)}g(i) \vec{a}(i)\} = dom (g) \cap dom(\vec{a})\cap dom(N)\cap dom(M) \supseteq^* C $.
\end{proof}

\begin{Theorem} \label{HAut-Projective-Limit-Theorem}
    Let $F \subseteq K_0\subseteq\cdots\subseteq K$ be a fully computable, infinite Galois extension in the sense of Definition \ref{def-fully-computably}.  Then, $$HAut ( \prod_C K/\prod_C F) \cong \varprojlim_{N \in \prod_C \Nb} HAut (\prod_{C,N} K_i / \prod_C F ) \cong \varprojlim_{N\in \prod_C \Nb }\prod_{C,N} G_i$$

    That is, the hyper-automorphism group of the cohesive power of an infinite Galois extension is isomorphic to the projective limit of the hyper-automorphism groups of the canonical pseudo-finite sub-extensions (\ref{sequence-extension-lemma}), which are in turn isomorphic to the pseudo-finite cohesive product of finite Galois groups.
\end{Theorem}

\begin{proof}

We construct the map $HAut( \prod_C K /\prod_C F) \to \varprojlim_{N \in \prod_C \Nb} HAut (\prod_{C,N} K_i / \prod_C F ) $ by sending $\sigma \in HAut (\prod_C K/ \prod_C F)$ to $\{ \sigma \upharpoonright \prod_{C,N} K_i \}_{N\in \prod_C \Nb}$.  

Suppose we have $ \sigma \in HAut (\prod_C K/ \prod_C F)$ with $\sigma \not= e$.  Then there is some $a \in \prod_C K$ with $\sigma (a) \not= a$, by Corollary \ref{corhaut}. Since $\prod_C K = \bigcup_{N \in \prod_C \Nb} \prod_{C,N} K_i$, $a \in \prod_{C,N} K_i$ for some $N$.  Thus $ \sigma \upharpoonright \prod_{C,N} K_i $ is non-trivial.  Therefore, the map is injective.

For surjectivity, let $\{ \tau_N \}_{N\in \prod_C \Nb} \in \varprojlim HAut(\prod_{C,N} K_i/\prod_C F)$.  First, note that each $\vec{a} \in \prod_C \mc{T}_K $ is in some $\prod_{C,N} T_{K_i}$ via \ref{sequence-extension-lemma}, as $T_K = \bigcup_{i\in \omega} T_{K_i} $.  Since each $\vec{a} \in \prod_C \mathcal{T}_{K}$ is contained in some $\prod_{C,N} \mathcal{T}_{K_i}$, we may take $\tau (\vec{a}) = \tau_N (\vec{a})$ for sufficiently large $N$. Via the compatibility for $\prod_{C,N} G_i$ from Lemma \ref{HAut-commutes-with-restriction}, $\tau(\vec{a})$ does not depend on the choice of $N$.  Further, $\tau(p(\vec{a}))=\tau_N(p(\vec{a}))=p(\tau_N(\vec{a}))=p(\tau(\vec{a}))$, for integral hyper-polynomials $p\in \prod_C \Zb [x_0,x_1,...]$, as $\tau_N \in HAut(\prod_{C,N} K_i /\prod_C F)$.  Thus $\tau \in HAut(\prod_CK/\prod_C F)$.

Clearly $\tau \mapsto \{ \tau_N\}_{N\in \prod_C \Nb}$.  Thus, the map is surjective, and hence an isomorphism.  Further, by Theorem \ref{Thm-HAut-is-Product}, $\varprojlim_{N \in \prod_C \Nb} HAut (\prod_{C,N} K_i / \prod_C F ) \cong \varprojlim_{N\in \prod_C \Nb }\prod_{C,N} G_i$.

\end{proof}

The purpose of the following lemma is to simplify notation.

\begin{Lemma} \label{We may freely conflate these three groups}
    Let $E/F$ be a finite Galois extension with Galois group $G$.  The following are canonically isomorphic, and the isomorphism commutes with the Galois action.

\begin{enumerate}
    \item $ G \curvearrowright E$

    \item $Gal( \prod_C E/\prod_C F) \curvearrowright \prod_C E$

    \item $\prod_C G \cong HAut(\prod_C E/ \prod_C F)$
\end{enumerate}

Moreover, the map $HAut(\prod_C E/ \prod_C F) \to Gal(\prod_C E/\prod_C F) $ is induced by ``forgetting'' the hyper-field structure on $\prod_C E$, and $Gal(\prod_C E/ \prod_C F) \to G$ is induced by restriction to $E$.

In a diagram:

% https://q.uiver.app/#q=WzAsMTAsWzAsMSwiSEF1dChcXHByb2RfQyBIL1xccHJvZF9DRikiXSxbMiwxLCJHYWwoXFxwcm9kX0MgSC9cXHByb2RfQ0YpIl0sWzQsMSwiR2FsKEgvRikiXSxbNCwzLCJIIl0sWzAsMywiXFxwcm9kX0MgSCJdLFsyLDMsIlxccHJvZF9DIEgiXSxbMSwwXSxbMSw0XSxbMCwwLCJIeXBlcmZpZWxkcyJdLFszLDAsIkZpZWxkcyJdLFszLDUsIiIsMCx7InN0eWxlIjp7InRhaWwiOnsibmFtZSI6Imhvb2siLCJzaWRlIjoiYm90dG9tIn19fV0sWzQsNV0sWzAsMSwiXFxjb25nIiwyLHsibGFiZWxfcG9zaXRpb24iOjMwfV0sWzEsMiwiXFxjb25nIiwyXSxbMCw0LCIiLDEseyJjdXJ2ZSI6NX1dLFsxLDUsIiIsMSx7ImN1cnZlIjo1fV0sWzIsMywiIiwxLHsiY3VydmUiOjV9XSxbNiw3LCIiLDIseyJzdHlsZSI6eyJib2R5Ijp7Im5hbWUiOiJkYXNoZWQifSwiaGVhZCI6eyJuYW1lIjoibm9uZSJ9fX1dLFs4LDksIlxcdGV4dHsgRm9yZ2V0ZnVsIEZ1bmN0b3J9IiwyXV0=
\[\begin{tikzcd}
	Hyperfields & {} && Fields & \\
	{HAut(\prod_C E/\prod_CF)} && {Gal(\prod_C E/\prod_CF)} && {Gal(E/F)} \\
	\\
	{\prod_C E} && {\prod_C E} && E \\
	& {}
	\arrow["{\text{ Forgetful Functor}}"', from=1-1, to=1-4]
	\arrow[dashed, no head, from=1-2, to=5-2]
	\arrow["\cong"'{pos=0.3}, from=2-1, to=2-3]
	\arrow[curve={height=30pt}, from=2-1, to=4-1]
	\arrow["\cong"', from=2-3, to=2-5]
	\arrow[curve={height=30pt}, from=2-3, to=4-3]
	\arrow[curve={height=30pt}, from=2-5, to=4-5]
	\arrow[from=4-1, to=4-3]
	\arrow[hook', from=4-5, to=4-3]
\end{tikzcd}\]

\end{Lemma}

\begin{proof}
    The groups are pairwise isomorphic, due to the fact that cohesive powers of finite structures are isomorphic and Theorem \ref{Thm-HAut-is-Product}.  The fact that the actions commutes follows from this.
\end{proof}

In the future, we will free conflate the above three groups for finite extensions $ E/F$.

In the classical setting, infinite Galois groups can be equipped with a metric structure.  The metric measures ``how much'' of $K$ one must check to determine if two automorphisms are distinct.

\begin{Definition}
    Let $G$ be the Galois group of the infinite Galois extension $K/F$, and let $K_i$ be finite Galois sub-extensions of $K$ such that $K= \bigcup K_i$ and $K_i \subseteq K_{i+1}$.  Given $\sigma, \tau \in G$ we take 

    $$d(\sigma, \tau) = \begin{cases}
        0 & \sigma = \tau \\
        2^{-n} & \text{ where } n \text{ is the least number such that $\sigma \upharpoonright K_n \not= \tau \upharpoonright K_n$}
    \end{cases}$$
\end{Definition}

\begin{Proposition}
    Let $K/F$ be a fully computable Galois extension in the sense of Definition \ref{def-fully-computably}. Let $\sigma, \tau \in Gal(K/F)$. Then $d(\sigma,\tau) < \frac{1}{2^n}$ is decidable uniformly in $\sigma, \tau$, and $n$.
\end{Proposition}

\begin{proof}
    Let $K_i,G_i$ be as in Definition \ref{def-fully-computably}, and let $\alpha_i$ be a primitive generator for $K_i$.  $d(\sigma,\tau) < \frac{1}{2^i}$ if and only if $\sigma(\alpha_i)= \tau(\alpha_i)$.  Note that $\alpha_i$ can be computed uniformly from $G_i$ by \ref{tfaegal}.
\end{proof}

Note that $d(\sigma, \tau) = 0$ is not decidable.  However, it is co-c.e.  It should be noted that the metric remembers facts about the presentation of the Galois extension, as opposed to its algebra. However, while the metric is unnatural for algebra, it is useful and interesting for cohesive powers.  

We define something resembling a hyper-metric on the hyper-automorphism group $HAut (\prod_C K / \prod_C F )$ for fully computable Galois extensions, which takes values in $\prod_C \Qb$.  

\begin{Definition} \label{def-hypermetric}

    Let $\sigma ,\tau \in HAut( \prod_C K/ \prod_C F)$, where K/F is a fully computable Galois extension. Set $$d ( \sigma,\tau) = \begin{cases}
        0 & \text{ if } \sigma=\tau; \\
        \frac{1}{2^N} & \text{ where }  N \text{ is the least element of $\prod_{C} \Nb$ such that $\sigma \upharpoonright \prod_{C,N}K_i \not= \tau \upharpoonright \prod_{C,N} K_i$},
        
        \end{cases}$$ where $\frac{1}{2^N}$ is understood to be an element of $\prod_C \Qb$. 
\end{Definition}

There is warranted concern that such a least $N$ may not in general exist, as $\prod_C \Nb$ has almost no induction (i.e., has definable proper cuts).  However, the situation is somewhat unique as a minimal $N$ does, in fact, exist.

\begin{Theorem} \label{Least-Nonstandard-N-Hypermetric}
Let $K/F$ be a fully computable Galois extension, as in Definition \ref{def-fully-computably}.  For each distinct pair $\sigma, \tau \in HAut(\prod_C K/ \prod_C F)$, there is a least $N$ such that $$\sigma_N = \sigma\upharpoonright \prod_{C,N} K_i \not= \tau\upharpoonright \prod_{C,N} K_i = \tau_N$$
\end{Theorem}

\begin{proof}

    In light of Theorem \ref{HAut-Projective-Limit-Theorem}, we will freely conflate $\sigma \in HAut(\prod_C K/\prod_C F)$ with its isomorphic image $$\{ \sigma_N \}_{N \in \prod_C \Nb} \in \varprojlim_{N\in \prod_{C} \Nb} HAut\left(\prod_{C,N} K_i / \prod_C F\right)$$

    Since $\sigma \not= \tau$, there is some $M \in \prod_C \Nb$ with $\sigma_M \not= \tau_M$.  Recall that by Theorem \ref{Thm-HAut-is-Product}, $\sigma_M, \tau_M \in \prod_{C,M} G_i$.  For each $i$, let $\alpha_i $ be a primitive generator of $K_i$.  Let $N(i)$ be the least number such that $\sigma_{M}(i)(\alpha_{N(i)} )\not= \tau_M(i)(\alpha_{N(i)})$.  Such a $N(i)$ exists for almost every $i$, as $\sigma_M(i) \not= \tau_M(i)$ for almost every $i$, and elements of the finite Galois groups $G_{M(i)}$ differ if and only if they act differently on $\alpha_{M(i)}$.  

    Thus, we have by construction that $\sigma_N = \sigma_M \upharpoonright \prod_{C,N} K_i \not= \tau_M \upharpoonright\prod_{C,N} K_i =\tau_N$.

    We will now show minimality.  It is sufficient to show that $\sigma_{N-1} =\tau_{N-1}$.  By minimality of $N(i)$, we have $\sigma_{N}(i) \upharpoonright K_{N(i)-1} = \tau_{N}(i) \upharpoonright K_{N(i)-1}$ for almost every $i$, as $\sigma_{N}(i)(\alpha_{N(i)-1})=\tau_{N}(i)(\alpha_{N(i)-1})$.  Therefore, for almost every $i$, we have $\sigma_{N-1}(i) = \sigma_{N}(i) \upharpoonright K_{N(i)-1} = \tau_{N}(i) \upharpoonright K_{N(i)-1}=\tau_{N-1}(i)$.  Thus, $\sigma_{N-1} = \tau_{N-1}$.
\end{proof}

\begin{Proposition}
    Let $\sigma, \tau, \rho \in HAut(\prod_C K/ \prod_C F)$. The hyper-metric of Definition \ref{def-hypermetric} satisfies the following:

    \begin{enumerate}
        \item $0\leq d(\sigma,\tau)$;
    
        \item $d(\sigma,\tau)= 0$ if and only if $\sigma=\tau$;

        \item $d(\sigma,\tau) = d( \tau, \sigma)$;

        \item $d( \sigma, \tau) \leq d(\sigma, \rho) + d(\rho, \tau) $.
    \end{enumerate}
\end{Proposition}

\begin{proof}
    \begin{enumerate}
        \item Trivial.
    
        \item Follows from Theorem \ref{Least-Nonstandard-N-Hypermetric}.

        \item Trivial.

        \item If $d(\sigma, \tau)$ is $0$, or is less than or equal to either $d(\sigma, \rho)$ or $d(\rho, \tau)$, then the result is immediate.  Therefore suppose both $d(\sigma, \rho)$ and $ d(\rho,\tau) $ are less than $ d(\sigma,\tau)$, which is non-zero.  Let $N \in \prod_C \Nb$ be such that $d(\sigma, \tau) = \frac{1}{2^N}$.  Since $d(\sigma,\rho) < \frac{1}{2^N}$, it must be the case that $\sigma$ agrees with $\rho $ on $\prod_{C,N} K_i$.  Similarly, $ \rho$ agrees with $\tau$ on $\prod_{C,N} K_i$.  Thus $\sigma$ and $\tau$ agree on $\prod_{C,N} K_i$, which implies that $d(\sigma,\tau) < \frac{1}{2^N}=d(\sigma,\tau)$, which is a contradiction.
    \end{enumerate}
\end{proof}

\begin{Definition}
    Let $K/F$ be a fully computable infinite Galois extension.  Take $$INF\left(\prod_C K/\prod_C F\right) = \left\{ \sigma \in HAut\left(\prod_C K/ \prod_C F \right) \middle| d(id,\sigma) =\frac{1}{2^N} \text{ for some } N \in \prod_C \Nb - \Nb \right\} $$

    In other words, $INF (\prod_C K / \prod_C F)$ is the set of infinitesimals in $HAut (\prod_C K/ \prod_C F)$.
\end{Definition}

\begin{Theorem}
    Let $K/F$ be a fully computable infinite Galois extension.  Then $INF (\prod_C K /\prod_C F)$ is a normal subgroup of $HAut (\prod_C K/ \prod_C F)$.  Furthermore, it consists of exactly the hyper-automorphisms that fix $ K$.  Thus, the quotient $HAut (\prod_C K /\prod_C F ) /INF (\prod_C K/ \prod_C F)$ is isometrically isomorphic to a dense subgroup of $Gal (K/F)$.
\end{Theorem}

\begin{proof}

Consider the map $\Psi: HAut(\prod_C K/ \prod_C F) \to Gal(K/F)$ that sends $\sigma$ to $\sigma \upharpoonright K$. Now, for any given hyper-automorphism $\sigma$,  $\sigma \in ker(\Psi)$ if and only if $\sigma \upharpoonright K_n \in G_n$ is trivial for every $n\in \omega$.  Recall that for $n \in \omega$, we have $\prod_{C,n} K_i \cong \prod_C K_n$. By Lemma \ref{We may freely conflate these three groups}, for any given $n$, we have that $\sigma \upharpoonright K_n$ is trivial if and only if $ \sigma\upharpoonright \prod_C K_n$, that is, $\sigma$ restricted to the cohesive \textit{power} of $K_n$ is the identity. This is true if and only if $d(id, \sigma) < \frac{1}{2^n}$. Thus, $ker \Psi = \{ \sigma \mid (\forall n\in \omega)\sigma \upharpoonright\prod_C K_n  =id \}=\{ \sigma \mid(\forall n\in \omega) d(id,\sigma) < \frac{1}{2^n} \}= INF(\prod_C K /\prod_C F)$.

To show density, by Lemma \ref{same fixed field iff same closure} it is sufficient to show that the fixed field of $Im(\Psi)$ is $F$.  Let $\beta \in K-F$.  Since $K= \bigcup_i K_i, $ it follows that $\beta \in K_n$ for some $n$.  Thus, there is some element $g\in G_n$ with $g(\beta) \not=\beta$.  By Lemma \ref{We may freely conflate these three groups}, $G_n \cong \prod_C G_n \cong HAut( \prod_C K_n/\prod_C F)$. Therefore, there is some $\sigma \in HAut(\prod_C K/\prod_C F)$ extending $g$, as $HAut(\prod_C K_n /\prod_C F)$ is a quotient of $HAut(\prod_C K/\prod_C F)$.  Thus, $(\Psi(\sigma))(\beta) =\sigma (\beta) = g(\beta) \not= \beta$.  Since $\beta$ was arbitrary, $Im(\Psi)$ has fixed field $F$.

Now, we show isometry.  To see this, we note that $\Psi(\sigma)$ acts non-trivially on $K_m$ if and only if $\sigma$ acts non-trivially on $K_m$, which is true if and only if $\sigma$ acts non-trivially on $\prod_C K_m$.  Suppose that $\sigma,\tau \in HAut(\prod_C K/\prod_C F)$ are such that $d(\sigma,\tau) = \frac{1}{2^n}$ for some $n\in \Nb$, a standard natural number. Then $\sigma$ and $\tau$ act differently on $\prod_C K_n$.  Thus, $\Psi(\sigma)$ and $\Psi(\tau)$ act differently on $K_n$, meaning $d(\Psi(\sigma),\Psi(\tau)) \geq \frac{1}{2^n}$.  Since $d(\sigma, \tau)< \frac{1}{2^{n-1}}$, we must have that $\sigma$ and $\tau$ act the same on $\prod_C K_{n-1}$. Therefore, they act the same on $K_{n-1}$.  Thus, $d(\Psi(\sigma),\Psi(\tau))<\frac{1}{2^{n-1}}$.  We conclude that $d(\Psi(\sigma),\Psi(\tau))= \frac{1}{2^n}=d(\sigma,\tau)$.
\end{proof}

We may think of $HAut(\prod_C K/\prod_C F)$ as some dense subgroup  of $Gal(K/F)$, together with a large group of infinitesimals that acts exclusively on the extension $\prod_C K -K$. 

Now, we turn our attention to the image of $HAut(\prod_C K /\prod_C F) \to Gal(\text{Alg}(\prod_C K)/\prod_C F)$. Theorem \ref{Cohesive Product of Actions Commutes with Quotients} gives us a characterization of this map already; however, the treatment of the kernel is unsatisfactory.  We now have the tools to give better description of Theorem \ref{Bad Description of the Kernel}.

\begin{Definition}
    Let $\{G_i\}_{i\in \omega}$ be a uniformly computable sequence of finite groups, and let $\{H_i\}_{i\in \omega}$ be a uniformly computable sequence of subgroups $H_i \leq G_i$.  We call the subgroup $\prod_C H_i $ an internal subgroup of $ \prod_C G_i$.
\end{Definition}

These subgroups are internal to the ``hyper-structure'' of the cohesive power.  

\begin{Theorem} \label{Internal Subgroups Finite Index}
    Let $\prod_C H_i \leq \prod_C G_i$ be an internal subgroup. Then  $\prod_C H_i$ has finite index $n$ if and only if $\{ i \mid [G_i : H_i] =n \} \supseteq^*C $.
\end{Theorem}

\begin{proof}
    Since the $H_i$'s are uniformly computable subgroups of finite groups, they are uniformly decidable subsets of the $G_i$'s. Therefore, we may augment $G_i$ with a predicate $\hat H$ such that $(G_i, \hat H) \models \hat H(g)$ iff $ g\in H_i$, while preserving uniform computability.  The statement ``the subgroup defined by $\hat H$ has index $n$'' is a Boolean combination of $\Sigma_1^0$ formulas with this predicate.  That is, $H_i$ has index $n$ if and only if 
    $$ (G_i, \hat H) \models \left(\exists g_1\dots  \exists g_n  \bigwedge_{j\not=k} \neg \hat H(g_jg_k^{-1}) \right) \wedge\left( \forall h_1\dots \forall h_{n+1} \bigvee_{j\not=k} \hat H(h_jh_k^{-1})\right).$$ The result now follows by the fundamental theorem of cohesive products.  
\end{proof}

\begin{Theorem}
    Let $\{K_i /F\}_{i\in \omega}$ be a uniformly computable sequence of finite Galois extensions with uniformly computable Galois actions $\{G_i \curvearrowright K_i\}_{i\in \omega}$.  There is a Galois correspondence between finite sub-extensions of $\text{\emph{Alg}}(\prod_C K_i)/\prod_C F$ and internal subgroups of $\prod_C G_i$.  That is, the finite sub extensions of $\text{\emph{Alg}}(\prod_C K_i)/\prod_C F$ are in an inclusion-reversing bijection with the internal subgroups of $\prod_C G_i$ of finite index.  %As before, \emph is just to make it plain text in the theorem enviroment
\end{Theorem}

\begin{proof}
    Let $(\prod_C F )(\alpha)/\prod_C F $ be a finite sub-extension of $\text{Alg}(\prod_C K_i)/\prod_C F$.  For each $i \in \omega$, let $H_i$ be the subgroup of $G_i$ that fixes $F(\alpha(i))$.  That is, $H_i = \{ g\in G_i \mid g(\alpha(i))= \alpha(i)\}$.  The $H_i$'s are computable subsets of the $G_i$'s, uniformly in $i$, as both $\alpha(i)$ and $G_i \curvearrowright K_i$ are uniformly computable.  Thus, $\prod_C H_i$ is an internal subgroup of $ \prod_C G_i$. 

    By the fundamental theorem of Galois theory, we have $[ G_i : H_i] = \deg_F (\alpha(i))$.  Furthermore, $\deg_F (\alpha(i)) = \deg_{\prod_C F} (\alpha)$ for almost every $i$ by Lemma \ref{classification of algebraic elements of cohprods}.  Thus, $[ G_i : H_i] = \deg_{\prod_C F} (\alpha)$ for almost every $i$, implying $[ \prod_C G_i : \prod_C H_i ]  = \deg_{\prod_C F} (\alpha)$. 

    Furthermore, $\prod_C H_i = \{ \tau \in HAut(\prod_C K_i /\prod_C F) \mid \tau(\alpha)=\alpha \}$, since for any $\tau$, we have $\tau(\alpha) = \alpha$ if and only if $\tau(i)(\alpha(i))=\alpha(i)$ for almost every $i$.  This happens if and only if $\tau(i) \in H(i)$ as $H(i)$ is defined to be the group with fixed field $F(\alpha(i))$.  Thus, $\tau \in \prod_C H_i$ if and only if $\tau$ fixes $\alpha$.  We conclude that $ H^{\prod_C F(\alpha)}=\prod_C H_i$, where $ H^{\prod_C F(\alpha)}$ is the subgroup of $HAut (\prod_C K_i / \prod_C F)$ that fixes $(\prod_C F)(\alpha)$.

    For the other direction, let $\prod_C H_i$ be an internal subgroup of $\prod_C G_i$ of finite index $n$.  Thus, for almost every $i$, $H_i \leq G_i$ has index $n$ by Theorem \ref{Internal Subgroups Finite Index}.  This means that the fixed field of $H_i$ is $F(\alpha(i))$ for some $\alpha(i)$ of degree $n$. Also, $\alpha(i)$ may be found computably.  Therefore, $\alpha \in \prod_C K_i$, and has degree $n$ over $\prod_C F$ by Lemma \ref{classification of algebraic elements of cohprods}.  Further, by construction, $ \prod_C H_i = \{ \tau \in \prod_C G_i \mid \tau(\alpha) = \alpha\}$. Thus, $\prod_C H_i$ is the subgroup of $\prod_C G_i$ that fixes $(\prod_C F)(\alpha)$.

    The fact that this is inclusion reversing follows analogously to the fundamental theorem of Galois theory.   
\end{proof}
\begin{Corollary} \label{third and fourth rows are exact}
    The map $HAut(\prod_C K_i /\prod_C F) \to Gal(\text{Alg}(\prod_C K_i)/\prod_C F)$ has kernel exactly the intersection of all internally finite subgroups of $\prod_C G_i$.  
    \end{Corollary}
\begin{proof}
    An element of $ HAut(\prod_C K_i / \prod_C F) $, which is canonically isomorphic to $\prod_C G_i$, is in the kernel of the map if and only if it fixes every finite sub-extension of $\text{Alg}(\prod_C K_i)$, which is true if and only if it is in the intersection of all internal subgroups of $\prod_C G_i $ of finite index.
\end{proof}
This final theorem is a compact summary of the analysis of the relationship between $HAut(\underline{\hspace{4mm}} ,\prod_C F)$ and $Gal(\text{Alg}(\underline{\hspace{4mm}}),\prod_C F)$.

\begin{Theorem}
    Let $K/F$, $\{K_i\}_{i\in \omega}$, $\{G_i\}_{i\in \omega}$ be a fully computable Galois extension, and let $N,M \in \prod_C \Nb$ with $N<M$.  Then, the following diagram commutes:

% https://q.uiver.app/#q=WzAsMTIsWzEsMCwiSEF1dChcXHByb2RfQyBLL1xccHJvZF9DIEYpIl0sWzIsMCwiR2FsKEFsZyhcXHByb2RfQyBLKSAvXFxwcm9kX0MgRikiXSxbMCwxLCJcXHZhcnByb2psaW0gXFxiaWdjYXBfe1xce0hfe04oaSl9XFx9fVxccHJvZF9DIEhfe04oaSl9Il0sWzEsMSwiXFx2YXJwcm9qbGltXFxwcm9kX3tDLE59IEdfaSJdLFsyLDEsIlxcdmFycHJvamxpbSBHYWwoQWxnKFxccHJvZF97QyxOfSBLX2kgKS9cXHByb2RfQyBGKSJdLFswLDIsIlxcYmlnY2FwX3tcXHtIX3tNKGkpfVxcfX1cXHByb2RfQyBIX3tNKGkpfSJdLFsxLDIsIlxccHJvZF97QyxNfSBHX2lcXGNvbmcgSEF1dChcXHByb2Rfe0MsTX0gS19pL1xccHJvZF9DIEYpIl0sWzIsMiwiR2FsKEFsZyhcXHByb2Rfe0MsTX0gS19pICkvXFxwcm9kX0MgRikiXSxbMCwzLCJcXGJpZ2NhcF97XFx7SF97TihpKX1cXH19XFxwcm9kX0MgSF97TihpKX0iXSxbMSwzLCJcXHByb2Rfe0MsTn0gR19pXFxjb25nIEhBdXQoXFxwcm9kX3tDLE59IEtfaS9cXHByb2RfQyBGKSJdLFsyLDMsIkdhbChBbGcoXFxwcm9kX3tDLE59IEtfaSApL1xccHJvZF9DIEYpIl0sWzAsMCwia2VyXFxQc2kiXSxbMCwxLCJcXFBzaSJdLFs1LDZdLFs2LDddLFs3LDEwXSxbOSwxMF0sWzYsOV0sWzUsOF0sWzgsOV0sWzExLDBdLFsxMSwyLCIiLDIseyJsZXZlbCI6Miwic3R5bGUiOnsiaGVhZCI6eyJuYW1lIjoibm9uZSJ9fX1dLFswLDMsIiIsMix7ImxldmVsIjoyLCJzdHlsZSI6eyJoZWFkIjp7Im5hbWUiOiJub25lIn19fV0sWzEsNCwiIiwwLHsibGV2ZWwiOjIsInN0eWxlIjp7ImhlYWQiOnsibmFtZSI6Im5vbmUifX19XSxbNCw3XSxbMyw2XSxbMiw1XSxbMiwzXSxbMyw0XV0=
\[\begin{tikzcd}
	{ker\Psi} & {HAut(\prod_C K/\prod_C F)} & {Gal(\text{Alg}(\prod_C K) /\prod_C F)} \\
	{\varprojlim \bigcap_{\{H_{N(i)}\}}\prod_C H_{N(i)}} & {\varprojlim\prod_{C,N} G_i} & {\varprojlim Gal(\text{Alg}(\prod_{C,N} K_i )/\prod_C F)} \\
	{\bigcap_{\{H_{M(i)}\}}\prod_C H_{M(i)}} & {\prod_{C,M} G_i\cong HAut(\prod_{C,M} K_i/\prod_C F)} & {Gal(\text{Alg}(\prod_{C,M} K_i )/\prod_C F)} \\
	{\bigcap_{\{H_{N(i)}\}}\prod_C H_{N(i)}} & {\prod_{C,N} G_i\cong HAut(\prod_{C,N} K_i/\prod_C F)} & {Gal(\text{Alg}(\prod_{C,N} K_i )/\prod_C F)}
	\arrow[from=1-1, to=1-2]
	\arrow[equals, from=1-1, to=2-1]
	\arrow["\Psi", from=1-2, to=1-3]
	\arrow[equals, from=1-2, to=2-2]
	\arrow[equals, from=1-3, to=2-3]
	\arrow[from=2-1, to=2-2]
	\arrow[from=2-1, to=3-1]
	\arrow[from=2-2, to=2-3]
	\arrow[from=2-2, to=3-2]
	\arrow[from=2-3, to=3-3]
	\arrow[from=3-1, to=3-2]
	\arrow[from=3-1, to=4-1]
	\arrow[from=3-2, to=3-3]
	\arrow[from=3-2, to=4-2]
	\arrow[from=3-3, to=4-3]
	\arrow[from=4-1, to=4-2]
	\arrow[from=4-2, to=4-3]
\end{tikzcd}\]
where the rows are left exact sequences.  We can imagine $0$'s on the left, as is the convention in commutative algebra.
\end{Theorem}

\begin{proof}
    The left exactness of the third and fourth rows follows from Theorem \ref{third and fourth rows are exact}. The left exactness of the second row follows from the left exactness of the third and fourth rows, as the projective limit functor is left exact. The isomorphism between $Gal(\text{Alg}(\prod_C K)/\prod_C F) $ and $\varprojlim_{N\in \prod_C \Nb} Gal(\text{Alg}(\prod_{C,N}K_i ) /\prod_C F$ follows from Lemma \ref{Infinite Galois Theory infinite proj lim}. The isomorphism between $HAut(\prod_C K/\prod_CF)$ and $\varprojlim_{N\in \prod_C \Nb} \prod_{C,N} G_i$ follows from Theorems \ref{Thm-HAut-is-Product} and \ref{HAut-Projective-Limit-Theorem}. The isomorphism from $ker \Psi$ to $\varprojlim \bigcap_{\{H_{N(i)}\}}\prod_C H_{N(i)}$ follows from the five lemma of homological algebra.  Thus the first row is left exact.  The commutativity now follows from diagram chasing.  
\end{proof}

\section*{Acknowledgments} The last three authors gratefully acknowledge support by FRG NSF grant DMS-2152095. Their research at the Hausdorff Institute for Mathematics, Bonn in Fall 2025, was partially funded by the Deutsche Forschungsgemeinschaft (DFG, German Research Foundation) under Germany's Excellence Strategy – EXC-2047/1–390685813.

\bibliographystyle{amsplain}
\bibliography{Paper}

\vfill

\end{document}